\documentclass[11pt]{amsart}

\usepackage{subfiles}

\usepackage{tabu}
\usepackage[margin=1in]{geometry}
\usepackage[dvipsnames]{xcolor}   		             		
\usepackage{graphicx}			
\usepackage{amssymb}
\usepackage{mathrsfs}
\usepackage{amsthm}
\usepackage{amsmath}
\usepackage{stmaryrd}
\usepackage{tikz}
\usepackage{tikz-cd}
\usepackage{accents}
\usepackage{upgreek}
\usepackage{enumerate}
\usepackage{bm}
\usepackage{mathtools}
\usepackage[all]{xy}
\usepackage{caption}

\usepackage{url}
\usepackage{float}
\usepackage{todonotes, derivative}
\usepackage{bbm}

\usepackage[full]{textcomp}
\usepackage[cal=cm]{mathalfa}
\usepackage{xparse}
\usepackage{comment}
\usepackage{cite}
\usepackage{xfrac}
\usepackage[normalem]{ulem}

\tikzset{
  commutative diagrams/.cd, 
  arrow style=tikz, 
  diagrams={>=stealth}
}

\theoremstyle{definition}

\makeatletter
\def\@tocline#1#2#3#4#5#6#7{\relax
  \ifnum #1>\c@tocdepth 
  \else
    \par \addpenalty\@secpenalty\addvspace{#2}%
    \begingroup \hyphenpenalty\@M
    \@ifempty{#4}{%
      \@tempdima\csname r@tocindent\number#1\endcsname\relax
    }{%
      \@tempdima#4\relax
    }%
    \parindent\z@ \leftskip#3\relax \advance\leftskip\@tempdima\relax
    \rightskip\@pnumwidth plus4em \parfillskip-\@pnumwidth
    #5\leavevmode\hskip-\@tempdima
      \ifcase #1
       \or\or \hskip 1em \or \hskip 2em \else \hskip 3em \fi%
      #6\nobreak\relax
    \dotfill\hbox to\@pnumwidth{\@tocpagenum{#7}}\par
    \nobreak
    \endgroup
  \fi}
\makeatother

\usetikzlibrary{calc}
\usetikzlibrary{fadings}
\usetikzlibrary{decorations.pathmorphing}
\usetikzlibrary{decorations.pathreplacing}

\newcounter{marginnote}
\DeclareMathAlphabet{\mathpzc}{OT1}{pzc}{m}{it}

\usepackage[backref=page]{hyperref}
\hypersetup{
  colorlinks   = true,          
  urlcolor     = blue,          
  linkcolor    = purple,          
  citecolor   = blue             
}

\theoremstyle{definition}
\newtheorem{theorem}{Theorem}[subsection]
\newtheorem{claim}[theorem]{Claim}

\newtheorem{corollary}[theorem]{Corollary}
\newtheorem{lemma}[theorem]{Lemma}
\newtheorem{proposition}[theorem]{Proposition}
\newtheorem{remark}[theorem]{Remark}

\newtheorem*{runningexample*}{Running example}

\newtheorem*{aside*}{Aside}

\newtheorem{definition}[theorem]{Definition}
\newtheorem{example}[theorem]{Example}

\newtheorem{proposition-definition}[theorem]{Proposition-Definition}

\newtheorem{maintheorem}{Theorem}
\newtheorem{maincorollary}[maintheorem]{Corollary}

\DeclareMathOperator{\ev}{\mathsf{ev}}

\newcommand{\Glog}{\mathbb{G}_\mathsf{log}}
\newcommand{\Gtrop}{\mathbb{G}_\mathsf{trop}}

\newcommand{\bcd}{\begin{center}\begin{tikzcd}}
\newcommand{\ecd}{\end{tikzcd}\end{center}}

\newcommand{\cM}{\mathcal{M}}
\newcommand{\cP}{\mathcal{P}}

\newcommand{\DR}{\mathsf{DR}}

\newcommand{\CC}{\mathbb{C}}

\newcommand{\NN}{\mathbb{N}}

\newcommand{\PP}{\mathbb{P}}

\newcommand{\RR}{\mathbb{R}}

\newcommand{\ZZ}{\mathbb{Z}}
\newcommand{\AAA}{\mathbb{A}}

\newcommand{\vir}{\text{\rm vir}}

\newcommand{\Ntrop}{{N}_{\mathsf{trop}}^{\Delta,\textbf{k}}}

\newcommand{\pt}{{\mathsf{pt}}}

\newcommand{\SetTropCurve}{T_{\Delta,p}^{\textbf{k}}}

\NewDocumentCommand{\compatibilitydatum}{m m m m m m O{} O{} O{}}{
\begin{equation*} \begin{tikzcd}[ampersand replacement=\&]
  \: \arrow{r} \& {#1} \arrow{r} \arrow{d}{#7} \& {#2} \arrow{r} \arrow{d}{#8} \& {#3} \arrow{r}{[1]} \arrow{d}{#9} \& \: \\
  \: \arrow{r} \& {#4} \arrow{r} \& {#5} \arrow{r} \& {#6} \arrow{r} \& \:
\end{tikzcd} \end{equation*}}

\NewDocumentCommand{\commutingsquare}{m m m m o O{} O{} O{} O{}}{
\begin{equation}\begin{tikzcd}[ampersand replacement=\&] \label{#5}
  #1 \arrow{r}{#6} \arrow{d}{#7} \& #2 \arrow{d}{#8} \\
  #3 \arrow{r}{#9} \& #4
\end{tikzcd}\IfValueTF{#5}{\label{#5}}{} \end{equation}}

\NewDocumentCommand{\cartesiansquare}{m m m m O{} O{} O{} O{}}{
\begin{equation*}\begin{tikzcd}[ampersand replacement=\&]
  #1 \arrow{r}{#5} \arrow{d}{#6} \arrow[dr, phantom, "\square"] \& #2 \arrow{d}{#7} \\
  #3 \arrow{r}{#8} \& #4
\end{tikzcd} \end{equation*}}

\NewDocumentCommand{\cartesiansquarelabel}{m m m m m O{} O{} O{} O{}}{
\begin{tikzcd}[ampersand replacement=\&]
  #1 \arrow{r}{#6} \arrow{d}{#7} \arrow[dr, phantom, "\square"] \& #2 \arrow{d}{#8} \\
  #3 \arrow{r}{#9} \& #4
\end{tikzcd}\IfValueTF{#5}{\label{#5}}{}
}

\NewDocumentCommand{\triangleofspaces}{m m m O{} O{} O{}}{
\begin{tikzcd} [ampersand replacement=\&]
#1 \arrow{r}{#4} \arrow[bend right]{rr}{#5} \& #2 \arrow{r}{#6} \& #3
\end{tikzcd}}

\begin{document}
 
\title{Tropical refined curve counting with descendants}
\author[Kennedy-Hunt]{Patrick Kennedy-Hunt}
\email{pfk21@cam.ac.uk}

\author[Shafi]{Qaasim Shafi}
\email{m.q.shafi@bham.ac.uk}
\author[Urundolil Kumaran]{Ajith Urundolil Kumaran}
\email{au270@cam.ac.uk}

\begin{abstract}
	We prove a $q$-refined tropical correspondence theorem for higher genus descendant logarithmic Gromov--Witten invariants with a $\lambda_g$ class in toric surfaces. Specifically, a generating series of such logarithmic Gromov--Witten invariants agrees with a $q$-refined count of rational tropical curves satisfying higher valency conditions. As a corollary, we obtain a geometric proof of the deformation invariance of this tropical count. In particular, our results give an algebro--geometric meaning to the tropical count defined by Blechman and Shustin. Our strategy is to use the logarithmic degeneration formula, and the key new technique is to reduce to computing integrals against double ramification cycles and connect these integrals to the non--commutative KdV hierarchy.
\end{abstract}

\maketitle
\vspace{-3em}
\tableofcontents

\vspace{-3em}

\section*{Introduction}
We study the logarithmic Gromov--Witten theory of toric varieties, relative their full toric boundary. 
Our results continue a line of inquiry which connects Hodge integrals \cite{MR1729095}, the local Gromov--Witten theory of curves \cite{MR2350052}, and refined tropical curve counting\cite{bousseau2018tropical}. Our main theorem is a tropical correspondence theorem which relates these logarithmic Gromov--Witten invariants of toric surfaces to refined tropical enumerative invariants introduced in work of Blechman and Shustin \cite{blechman2019refined}.
The Gromov--Witten theory of threefolds is particularly interesting: it connects to sheaf-counting theories, and the dimension of the moduli spaces associated to the enumerative count are independent of genus. Our theorem can be viewed as a result for the ``local surface threefolds'' given by $X\times\mathbb A^1$, where $X$ is a toric surface.

Fix a multiset $\Delta^\circ$ of vectors in $\ZZ^2 \setminus \{(0,0)\}$ with sum zero, together with non-negative integers $n$ and $k_1,\dots,k_n$ such that $$n-1 + |\Delta^\circ| = 2n + \sum\limits_{i=1}^{n} k_i.$$ Associated to this discrete data are two enumerative invariants. 

\begin{enumerate}[(1)]
    \item The datum $\Delta^\circ$ determines a toric surface $X_{\Delta}$ and a curve class $\beta_{\Delta}$ on this toric surface. For every genus $g\geq0$, there is an associated \textit{logarithmic Gromov--Witten }invariant with $\lambda_g$ insertion defined as the following intersection product
\begin{equation*}
   {N}_{g,\Delta}^{\textbf{k}} =\int_{[\mathsf{M}_{g,\Delta}]^{\vir}} (-1)^{g} \lambda_g \prod_{i=1}^{n} \ev_i^{\star}(\pt) \,  \psi_{i}^{k_i}.
\end{equation*} Our notation is defined in Section \ref{sec:logGW}. This logarithmic Gromov--Witten invariant captures information about algebraic curves passing through a generic collection of $n$ points in $X_{\Delta}$ subject to stationary descendant constraints. The multiset $\Delta^\circ$ controls the tangency constraints as well as the curve class corresponding to this Gromov--Witten invariant.
\item Fixing a generic ordered tuple of $n$ points $p = (p_1,\ldots,p_n)$ in $\mathbb{R}^2$, the data $(\Delta^\circ,\textbf{k})$ defines a finite set of genus zero tropical curves $T_{\Delta,p}^{\textbf{k}}$. To each tropical curve $h \in T_{\Delta,p}^{\textbf{k}}$ we assign a rational function $m_h(q)$ of formal variable $q^{1/2}$. We define a count of tropical curves $$\Ntrop(q) =  \sum_{h \in T_{\Delta,p}^{\textbf{k}}} m_h(q).$$ The rational function $\Ntrop(q)$ is an invariant defined in terms of polyhedral geometry.
\end{enumerate}

 A precise statement of the following theorem appears in Section \ref{sec:MainTheorem}.

\begin{maintheorem}\label{thm:main}
    After the change of variables $q=e^{iu}$ we have the equality
    \begin{equation*}\label{eq:maintheorem}
   \sum\limits_{g \geq 0} {N}_{g,\Delta}^{\textbf{k}} \, u^{2g - 2 + |\Delta^{\circ}| - \sum_{i} k_i} = \Ntrop(q).
\end{equation*}

\end{maintheorem}

The tropical curve count on the right of Theorem \ref{thm:main} can be computed in two steps. \begin{enumerate}[(1)] \item A combinatorial process to enumerate all tropical curves.
\item Compute the multiplicity of each tropical curve as a product over its vertices. \end{enumerate} The first step is possible through floor diagrams, lattice path algorithms \cite{Milk} or with software \cite{corey2022counting}. Although tropical curves are studied in higher dimensions, we emphasise that, in dimension two, these techniques provide an efficient method to enumerate all tropical curves. See ~\cite[Appendix, p.26-28]{blechman2019refined} for an explicit computation of $\Ntrop$ in a particular example using a lattice path algorithm. 

Remarkably, the tropical invariants that arise in Theorem \ref{thm:main} were discovered through combinatorial considerations by Blechman--Shustin. Our work sheds light on several features. Notably, a priori $\Ntrop(q)$ depends on the choice of points $p$ in $\mathbb{R}^2$. The left hand side of the equation in Theorem \ref{thm:main} does not depend on $p$. Blechman and Shustin showed that this deformation invariance imposes a severe restriction on $m_h(q)$ and were thus able to write down a multiple of $m_h(q)$ and so a multiple of $\Ntrop$, without establishing its relation to Gromov--Witten invariants. Theorem \ref{thm:main} thus provides a new proof of the following result of Blechman and Shustin \cite{blechman2019refined}.

\begin{maincorollary}\label{corr:TropDefInv}
The count of tropical curves $\Ntrop$ is independent of the choice of points $p$.
\end{maincorollary}

Our results suggest a generalisation of the Blechman--Shustin multiplicity to tropical curves of genus greater than zero, see Section \ref{sec:GenShustBlech}. The count of genus $g$ tropical curves with this generalised multiplicity is independent of $p$.

\subsection{Tropical correspondence theorems}

Tropical geometry is a combinatorial shadow of algebro-geometric problems, well suited to capturing enumerative information. Mikhalkin pioneered the connection between tropical and enumerative geometry by establishing an equality between  counts of algebraic curves in toric surfaces of fixed degree and genus and a weighted count of tropical curves, of the same degree and genus \cite{Milk}. A number of subsequent results have exhibited this correspondence principle \cite{gathmann2007gromov, gathmann2008kontsevich, cavalieri2010tropical, cavalieri2011wall, cavalieri2016tropicalizing, MandelRuddat, graefnitz2022tropical,NishSieb06}. Notably, Nishinou and Siebert \cite{NishSieb06} applied degeneration techniques to generalise Mikhalkin's result to counts of rational curves in toric varieties of arbitrary dimension. Contemporary results often connect tropical geometry to logarithmic Gromov--Witten invariants, which are readily accessible through these degeneration techniques.

A logarithmic Gromov--Witten invariant is defined as an intersection product on the moduli space of stable logarithmic maps and thus tautological cohomology classes on this moduli space may be used as \textit{insertions} to define new invariants. Two flavours of tautological cohomology classes play an important role in the sequel: Chern classes of cotangent line bundles denoted $\psi_i$ and the top Chern class of the Hodge bundle denoted $\lambda_g$. We recall the definition of these classes in Section \ref{sec:ModuliOfCurves}.

Mikhalkin suggested that under a correspondence theorem $\psi$-classes should correspond to counts of tropical curves with high valence vertices \cite{mikhalkin2006moduli}. Genus zero correspondences with descendants are known for $\PP^2$, $\PP^1$ and $\PP^1 \times \PP^1$ \cite{markwig2009tropical, gross2010mirror,cavalieri2017tropicalP1,rau2016intersections}, with the most general results coming from \cite{MandelRuddat,MR3816384}. More recently, there have been results for $\lambda_g$ class insertions, but the relationship is more complicated and takes the form of a \emph{refined tropical correspondence}. The significance of the $\lambda_g$ class is that it allows us to pass from curve counting on a toric surface $X$, to curve counting on the associated local Calabi--Yau threefold $X \times \mathbb{A}^1$. On any threefold, the virtual dimension of the mapping space is independent of the genus, and in this case the associated invariants are exactly the logarithmic Gromov--Witten invariants of $X$ with a $\lambda_g$ insertion. By packaging these invariants in an appropriate generating function, one obtains refined curve counts.

\emph{Refined} tropical correspondence theorems are an example of a ubiquitous phenomena in mathematics: \emph{quantum} analogues to classical results. Such a result depends on a parameter $q$ which recovers the classical result as $q \rightarrow 1$. In the setting of tropical correspondence theorems this began with \cite{block2016refined}. The authors gave multiplicities of plane tropical curves, depending on a parameter $q$ which refined the ordinary multiplicity appearing in the traditional tropical correspondence theorems. In classical situations, such as Severi degrees, enumerative invariants can be calculated via Euler characteristics of relative Hilbert schemes of points on planar curves. The authors of \cite{gottsche2014refined} suggested that refined tropical invariants corresponded to Severi degrees with $\chi_y$ genus in place of the Euler characteristic, see also \cite{nicaise2018tropical}. 

Correspondence results involving $\lambda_g$ class insertions and higher genus logarithmic Gromov--Witten theory appeared later in work of Bousseau \cite{bousseau2018tropical}. This provides an alternative perspective on the algebro-geometric information encoded in $q$-refined tropical curve counts. Theorem \ref{thm:main} is parallel to the work of Bousseau, in that we provide the connection between logarithmic Gromov--Witten theory and the tropical curve counts of  Blechman and Shustin \cite{blechman2019refined}, who extended the refined counts of Block--G\"ottsche to plane tropical curves with higher valence vertices.

\subsection{Relationship to literature} The contribution of the present paper is to handle the simultaneous presence of $\lambda_g$ and $\psi$ class insertions. The place of our result in the literature may be summarised with the following diagram.
$$
\begin{tikzcd}
                                         & \text{\cite{bousseau2018tropical}} \arrow[d, "(1)", no head]     &                \\
\text{\cite{blechman2019refined,gottsche2019refined}} \arrow[r, "(2)", no head] & \textrm{Present Paper} \arrow[r, "(3)", no head] & \text{\cite{MR3816384,markwig2009tropical,MandelRuddat}}
\end{tikzcd}$$

\begin{enumerate}[(1)]
    \item We are not aware of how to generalise Bousseau’s argument to the descendant setting, so we must perform certain atomic calculations directly. As a result we provide a new proof of the special case that there are no descendants: a theorem proved in Bousseau's work. Our proof uses three key ingredients. First, the degeneration formula for logarithmic Gromov--Witten theory. This was already used in Bousseau's work, but new subtleties appear in the descendant setting. Secondly, the connection between logarithmic Gromov--Witten invariants of toric surfaces and double ramification cycles \cite{AjithDhruv}. Finally, we use the connection between intersections against double ramification cycles and the KdV hierarchy\cite{Buryak,buryak2021quadratic}. 
    \item Our theorem provides an algebro-geometric interpretation of the tropical count defined by Blechman and Shustin. Their invariants, though purely combinatorial, are part of a natural system that contain the geometric refinements of rational curve counts. The generalization of Blechman--Shustin’s work to higher genus is a topic of ongoing interest.
    \item Due to the nature of our correspondence theorem, the tropical curves arising in our computations coincide with the ones considered by Markwig--Rau \cite{markwig2009tropical}. Our multiplicities coincide with Markwig--Rau’s multiplicities when $q$ approaches $1$. This is a combinatorial statement. From a geometric viewpoint, the $q \rightarrow 1$ specialization simply recovers the correspondence theorems of \cite{MandelRuddat,MR3816384}, in the special case of plane curves. The multiplicity of the tropical curves in our theorem split as a product over multiplicities assigned to vertices. This contrasts with the general case of \cite{MandelRuddat,MR3816384}.
\end{enumerate}

{\subsection{Proof strategy: double-ramification cycles and integrable hierachies}

 Degeneration arguments and logarithmic intersection theory allow us to build on \cite{gross2023remarks,ranganthan2022logarithmic} to prove a simple degeneration formula in our setting. This is parallel to the degeneration arguments of \cite{NishSieb06,bousseau2018tropical,MandelRuddat} but additional subtleties arise due to the intersection of the $\lambda_g$ and $\psi$ conditions. This reduces the proof of Theorem \ref{thm:main} to computing intersection products of the form $$\int_{[\mathsf{M}_{g,\Delta}]^{\vir}} \lambda_g \ev^{\star}(\pt) \,  \psi^{k}$$
which we call \textit{vertex contributions.} 
Here $\mathsf{pt}$ denotes the cohomology class poincare dual to a generic point.
These vertex contributions are also descendant logarithmic Gromov--Witten invariants of a toric surface with at most one $\lambda_g$ class, albeit simpler, with a single point insertion and power of a $\psi$ class. 

In joint work with Ranganathan \cite{AjithDhruv} the third author proved a result implying that vertex contributions could be expressed as intersection products on the moduli space $\overline{\mathcal{M}}_{g,n}$ of genus $g$ curves $$\int_{[\mathsf{M}_{g,\Delta}]^{\vir}} \lambda_g \ev_1^{\star}(\pt) \,  \psi_{1}^{k} = \int_{\overline{\mathcal{M}}_{g,n+1}} \lambda_g \,  \psi_{1}^{k} \mathsf{TC}_{g}(\Delta).$$ The class $\mathsf{TC}_{g}(\Delta)$ is the \textit{toric contact cycle,} a higher rank generalisation of the double ramification cycle $\mathsf{DR}_g(\textbf{a})$. The toric contact sometimes appears in literature, under the moniker \textit{the double double ramification cycle}, see \cite{MR4490707,molcho2021case,holmes2022logarithmic} for background. Both classes are recalled in Section \ref{Section : LogGW and DR}. 

Let $\Delta$ be the $2 \times (n+r)$ matrix defined by setting the first $n$ columns zero and the final columns to be the elements of $\Delta^\circ$. When computing vertex contributions we may assume $n=1$. Writing $\Delta^x, \Delta^y$ for the rows of the matrix $\Delta$, we establish $$\lambda_g \cdot \mathsf{TC}_{g}(\Delta) = \lambda_g \cdot \mathsf{DR}_g(\Delta^x) \cdot \mathsf{DR}_g(\Delta^y)$$ in Proposition \ref{DDRvsDRDR}. The naive hope that $\mathsf{DR}_g(\Delta^x) \cdot \mathsf{DR}_g(\Delta^y) = \mathsf{TC}_g(\Delta)$ is false, and only true on the compact type locus, so correction terms are required. The content of Proposition \ref{DDRvsDRDR} then is that $\lambda_g$ annihilates the correction terms. See also \cite{holmes2019multiplicativity} for a closely related statement.
We are left to compute integrals of the form 
\begin{equation}
I_{g,d}(a_1,\ldots,a_n;b_1,..,b_n) = \int_{\overline{\mathcal{M}}_{g, n+1}} \lambda_g \psi_1^d \DR_g(0, a_1, \ldots, a_{n}) \DR_g(0, b_1, \ldots, b_{n}).
\end{equation}
These integrals are computed with techniques from the theory of integrable hierarchies. Buryak \cite{Buryak} constructed the double ramification hierarchy, whose hamiltonians are generating functions of these integrals. This double ramification hierarchy coincides with the non-commutative KdV hierarchy \cite{buryak2021quadratic}. This KdV hierarchy is well understood, allowing us to write explicit formulas for generating functions of $I_{g,d}(a_1,\ldots,a_n;b_1,..,b_n)$.

\begin{figure}[ht]
    \centering
    \includegraphics[width=9cm]{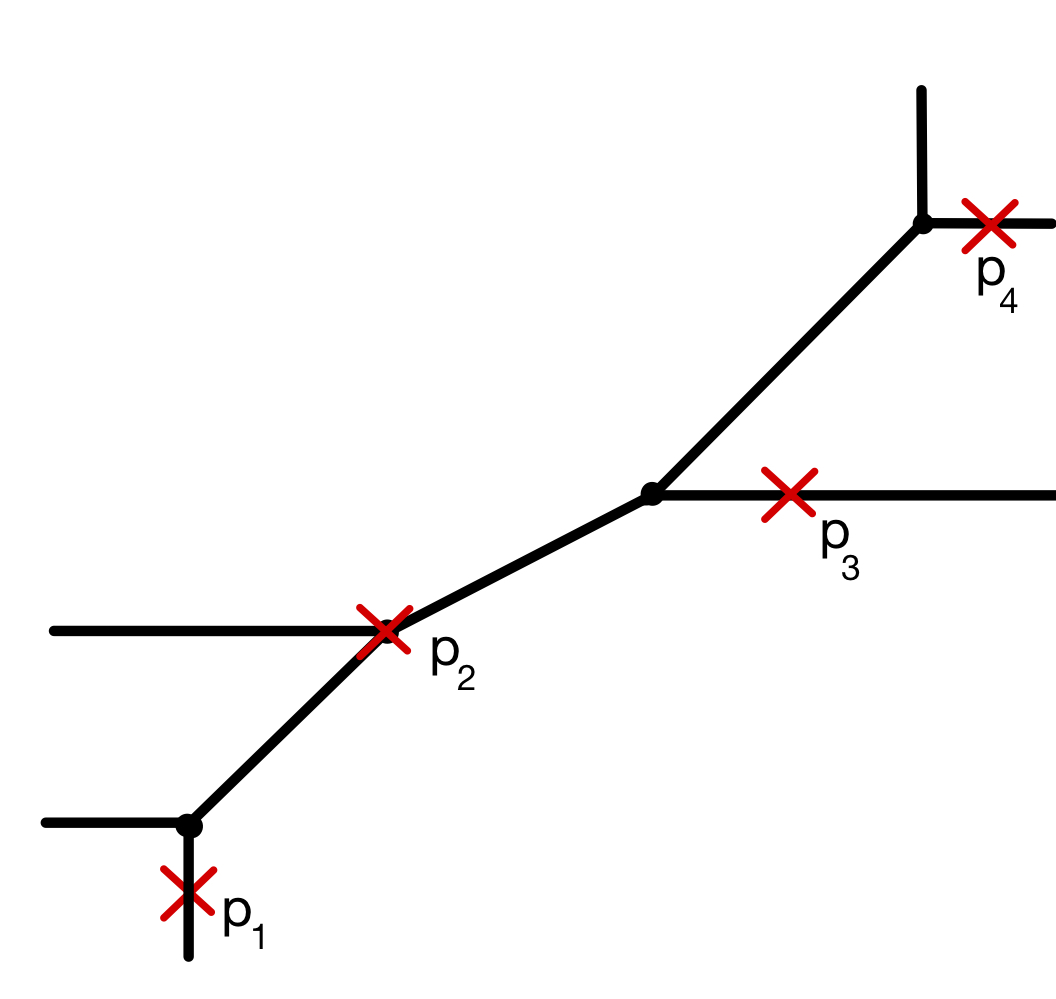}
    \caption{ Above is a tropical curve including edges with primitive directions $(1,1)$ and $(2,1)$. }\label{fig:ExampleCount}
\end{figure}

\begin{example} Set $$\Delta^\circ = \{(-1,0),(-1,0),(0,-1),(1,0),(1,0),(0,1)\}, \quad \mathbf{k}= (0,1,0,0)$$ and fix points $(p_1,p_2,p_3,p_4)= p$. See Figure \ref{fig:ExampleCount} for a tropical curve contributing to $\Ntrop$ for a certain choice of $p_i$. The multiplicity of this tropical curve is a product over vertices. The vertex marked by $p_2$ contributes $\frac{1}{2}(q^{1/2}+q^{-1/2})$ and all other vertices contribute $(-i)(q^{1/2}-q^{-1/2})$. This tells us the total contribution of this tropical curve is $\frac{i}{2}(q^{1/2}+q^{-1/2})(q^{1/2}-q^{-1/2})^3$.
\end{example} 
\subsection{Future directions} {The long term hope is to compute the descendant partition function of Gromov--Witten invariants of toric threefolds pairs in all genus. To achieve this goal, two generalisations of current results are required: passing to honest threefolds, rather than local surfaces, and higher genus tropical curves. Parker established the primary part of the above computation in the setting of his theory of exploded manifolds \cite{Parker}. Even assuming the equivalence of logarithmic and exploded invariants, our results are new because we handle descendants.}

\subsubsection{Logarithmic Gromov--Witten/Donaldson--Thomas invariants} The $q = e^{iu}$ change of variables appearing in Theorem \ref{thm:main} is the same change of variables that controls the logarithmic Gromov--Witten/Donaldson--Thomas correspondence \cite{MR20,MR23}. Including descendants in the logarithmic Gromov--Witten/Donaldson--Thomas correspondence is a subtle problem, studied intensely by Moreira, Okounkov, Oblomkov, and Pandharipande \cite{MR4479829,MR4076076}. Since the Donaldson--Thomas/ Pandharipande--Thomas wall-crossing is expected to be trivial in this setting, our calculations offer a concrete and testable prediction for these new conjectures.
\subsubsection{Higher starting genus} 
Fix a multiset $\Delta^\circ$ of vectors in $\ZZ^2 \setminus \{(0,0)\}$ with sum zero, together with non-negative integers $g_0, n$ and $k_1,\dots,k_n$ such that $$n + g_0 -1 + |\Delta^\circ| = 2n + \sum\limits_{i=1}^{n} k_i.$$ 
For every genus $g\geq g_0$ we consider the following logarithmic Gromov--Witten invariant
\begin{equation*}
   {N}_{g,g_0,\Delta}^{\textbf{k}} =\int_{[\mathsf{M}_{g,\Delta}]^{\vir}} (-1)^{g-g_0} \lambda_{g-g_0} \prod_{i=1}^{n} \ev_i^{\star}(\pt) \,  \psi_{i}^{k_i}.
\end{equation*} The present paper establishes a tropical correspondence theorem in the the case $g_0=0$. The case $g_0>0$ without descendant insertions is known \cite{bousseau2018tropical}. The difficulty with higher genus tropical curves of higher valence is that such curves may be superabundant~\cite[Remark 2.6]{MandelRuddat}.  Consequently, the degeneration and gluing arguments of Section \ref{Section : Decomposition and Gluing} fail.

\subsubsection{Weak Frobenius Structure Conjecture} In Gross--Hacking--Keel's construction of the mirror to a log Calabi--Yau surface $X$ \cite{gross2015mirror}, the mirror is constructed as the spectrum of an algebra of \textit{theta functions}: the first example of a theta function is the unit. The \emph{weak frobenius structure conjecture} asserts that the coefficient of the unit in any product of theta functions is a sum of genus zero descendant logarithmic Gromov--Witten invariants of $X$. The conjecture is known for cluster varieties \cite{mandel2021theta}, and  Looijenga pairs satisfying \cite[Assumptions 1.1]{gross2022canonical} by \cite{johnston2022comparison}. 

A deformation quantization of the Gross--Hacking--Keel mirror, depending on a parameter $q$, is known \cite{bousseau2020quantum}. Once again this involves the construction of an algebra, this time non-commutative, generated by \textit{quantum theta functions}. Theorem \ref{thm:main} provides an avenue to explore the connection between the unit term in any product of quantum theta functions, and descendant logarithmic Gromov--Witten invariants with a $\lambda_g$ insertion. Products of quantum theta functions are computed with quantum scattering diagrams, which one expects are related to refined counts of tropical curves.  
Upcoming work of Gr\"afnitz, Ruddat, Zaslow and Zhou offers progress in this direction \cite{GRZZ}.

\subsubsection{Generalised Blechman--Shustin}\label{sec:GenShustBlech} Our results suggest a generalisation of the multiplicities of Blechman and Shustin to higher genus tropical curves. Indeed, the degeneration arguments of Section \ref{Section : Decomposition and Gluing} show that there is a tropical correspondence result for the logarithmic Gromov--Witten invariants  $${N}_{g,g_0,\Delta}^{\textbf{k}} =\int_{[\mathsf{M}_{g,\Delta}]^{\vir}} (-1)^{g-g_0} \lambda_{g-g_0} \prod_{i=1}^{n} \ev_i^{\star}(\pt) \,  \psi_{i}^{k_i}.$$ On the tropical side one sums over tropical curves of genus $g_0$. One can define the generalised Blechman--Shustin multiplicity as the multiplicity of a tropical curve under such a correspondence theorem. Theorem \ref{thm:main} shows that the generalised Blechman--Shustin multiplicity coincides with the Blechman--Shustin multiplicity in the case $g_0=0$. A version of Corollary \ref{corr:TropDefInv} shows counting tropical curves with the generalised 
Blechman--Shustin multiplicity is independent of the choice of $p$. This argument does not suggest that the generalised Blechman--Shustin multiplicity of a tropical curve $\Gamma$ will be a product over vertices of $\Gamma$.
\subsection{Acknowledgements and funding information}\label{sec:acknowledgements} The authors are grateful to their common PhD advisor Dhruv Ranganathan for numerous helpful conversations, suggesting the project and without whom this project would not have happened. We thank Thomas Blomme for helpful conversations which established the link between our multiplicity and the multiplicity of Blechman and Shustin, we especially thank Thomas for communicating the proofs of Proposition \ref{prop:Link To BS mult} and Lemma \ref{lem:DiffMult}. Finally, we thank Paulo Rossi for a valuable email exchange as well as Sam Johnston and Danilo Lewanski for numerous helpful conversations. P.K.-H. is supported by an EPSRC Studentship, reference 2434344. A.U.K. is supported by an EPSRC Studentship, reference 2597628. Q.S. is supported by UKRI Future Leaders Fellowship through
grant number MR/T01783X/1.

\subsection{Data availability statement} Data sharing not applicable to this article as no datasets were generated or analysed.

\subsection{Conflicts of interest} There are no conflicts of interest to declare. See Section \ref{sec:acknowledgements} for funding information.

\section{Tropical enumerative geometry}\label{Section : Tropical}
In this section we set up the tropical enumerative problem. Fix $\Delta^\circ = \{\delta_1,\ldots,\delta_r\}$ a multiset of non-zero vectors in $\mathbb{Z}^2$ with sum zero. Fix also an ordered tuple $\textbf{k} = (k_1,\ldots,k_n)$ of $n$ non--negative integers.  In the main body, $\Delta^\circ$ will record the tangency order of curves with the boundary and $\mathbf k$ will record the $\psi$ class power attached to each of the $n$ marked points.

In the sequel we partially record these data in a $2\times (r+n)$ matrix denoted $\Delta$. The first $n$ columns are zero and the final $r$ columns are the vectors $\delta_i$. In writing down such a matrix we implicitly choose an order on the $\delta_i$; we fix this order without further comment.  

\subsection{First definitions} Define the \textit{lattice length $\ell$} of a vector $\delta_i\in \mathbb{Z}^2$ to be the maximal positive integer $\ell$ such that one can write $\delta_i = \ell \overline{\delta}_i$ for $\overline{\delta}_i$ a vector in $\mathbb{Z}^2$ called the \textit{direction} of $\delta_i$. 

\subsubsection{Tropical curves} Our definition of tropical curve coincides with the definition presented in \cite[Section 2.3]{bousseau2018tropical}. We refer the reader to \cite{abramovich2020decomposition,MandelRuddat,Milk,NishSieb06} for background.

\begin{definition}
    A graph $\Gamma$ is a triple consisting of 
    \begin{enumerate}[(1)]
        \item a finite set of vertices $V(\Gamma)$;
        \item a finite set $E_f(\Gamma)$ of pairs of elements of $V(\Gamma)$ called the \textit{bounded edges};
        \item and a multiset $E_{\infty}(\Gamma)$ of elements of $V(\Gamma)$ called the \textit{unbounded edges}.
    \end{enumerate}
    For us, all graphs are connected. This means for any two elements $v,w$ of $V(\Gamma)$ there exists a sequence of elements $v=u_1,\ldots,u_k=w$ such that $\{u_i,u_{i+1}\}$ lies in $E_f(\Gamma)$ for all $i$. 
\end{definition}
An \textit{abstract tropical curve} $|\Gamma|$ is the underlying topological space of a graph $\Gamma$.

\begin{definition}
    A parametrised tropical curve ${h}:{\Gamma} \rightarrow \mathbb{R}^2$ consists of the following data.
    \begin{enumerate}[(1)]
        \item A graph $\Gamma$ and a non--negative integer $g_V$ assigned to each vertex $V$ of $\Gamma$ called the \textit{genus}.
        \item A bijective function $$L: E_\infty(\Gamma)\rightarrow \{1,...,r+n\}.$$
        \item A \textit{vector weight} $v_{V,E} \in \mathbb{Z}^2$ for every edge--vertex pair $(V,E)$ with $E\in E_f(\Gamma)\cup E_\infty(\Gamma)$ and $V\in E$ such that for every vertex $V$, the
following balancing condition is satisfied:
$$\sum_{E: V \in E} v_{V,E}=0.$$
\item For each bounded edge $E \in E_f(\Gamma)$ a positive real number $\ell(E)$,
called the \textit{length} of $E$.
\item A map of topological spaces $h: |\Gamma| \rightarrow \mathbb{R}^2$ such that restricting $h$ to the edge $\{v_1,v_2\}$ is affine linear to the line segment connecting $h(V_1)$ and $h(V_2)$ and moreover 

$$h(V_2) - h(V_1) = \ell(E)v_{V_1,E}.$$
Also, restricting $h$ maps an unbounded edge $E$ associated to vertex $V$ to the ray $h(V) + \mathbb{R}_{\geq 0}v_{V,E}$.
    \end{enumerate}
We say $h$ has \textit{degree} $\Delta$ if $v_{V,E}$ coincides with the $L(E)^\mathrm{th}$ column of $\Delta$ whenever $E \in E_\infty (\Gamma)$. The \textit{genus} of a parametrised tropical curve is obtained by adding the sum of $g_V$ to the Betti number of $|\Gamma|$. The \textit{weight} of an edge $E$, denoted $w(E)$, is the lattice length of $v_{V,E}$.
\end{definition}

For a vertex $V$ of $\Gamma$ write $E_\infty^+(V)$ for the set of unbounded edges $E$ adjacent to $V$ such that $v_E\ne 0$ and $E_f(V)$ for the set of bounded edges adjacent to $V$. The \textit{valency} $\mathsf{val}_V$ of a vertex $V$ is the cardinality of $E_f(V)\cup E_\infty^+(V)$. Write $\Delta_V^\circ$ for the multiset of all $v_{V,E}$ for fixed $V$.

\subsection{Multiplicities}\label{sec: multiplicities} In the sequel we count parametrised genus zero tropical curves of degree $\Delta$ satisfying certain incidence conditions. Tropical curves are counted with a multiplicity, closely related to the multiplicity of \cite{blechman2019refined}. This multiplicity is a product over multiplicities assigned to each vertex of our tropical curve. Fix for the remainder of the section a parametrised tropical curve ${h}:\Gamma \rightarrow \mathbb{R}^2$.

We will be counting tropical curves passing through a tuple of points $p=(p_1,\ldots,p_n)$ in $\mathbb{R}^2$, and thus vertices of $\Gamma$ come in two flavours. A vertex is \textit{pointed} if its image under $h$ coincides with one of the $p_i$. Vertices which are not pointed are \textit{unpointed}.
\subsubsection{Notation} Given two elements $v_1,v_2 \in \mathbb{Z}^2$ we define $v_1\wedge v_2$ to be the determinant of the matrix with first column $v_1$ and second column $v-2$. Following \cite{blechman2019refined}, define functions of $q$ $$[v_1\wedge v_2]_+ = q^{\frac{1}{2}v_1\wedge v_2}+q^{-\frac{1}{2}v_1\wedge v_2}\quad \quad [v_1\wedge v_2]_- = q^{\frac{1}{2}|v_1\wedge v_2|}-q^{-\frac{1}{2}|v_1\wedge v_2|}.$$

The cyclic group with $N$ elements acts on the set of ordered tuples of $N$ distinct elements from the set $\{1,\ldots,N\}$. The action is induced by sending the integer in position $i$ to position $i+1$ mod $N$. The set of orbits of this action is the set $\Omega_N$ of \textit{cyclic permutations}. We sometimes write $\Omega_N=\Omega_N(a_1,\ldots,a_N)$ when we wish to think of $\Omega_N$ as the set of cyclic permutations of vectors $(a_1,\ldots,a_N)$. Where no confusion is likely we omit $(a_1,\ldots,a_N)$ from the notation. For $\omega$ a cyclic permutation, choose an ordered tuple $\tilde{\omega}$ in the orbit $\omega$. Define $$k(\omega) = \sum_{2\leq i < j \leq N}a_{\tilde{\omega}(i)} \wedge a_{\tilde{\omega}(j)}$$ where $\tilde{\omega}(i)$ sends $i$ to the element in the $i^{\mathsf{th}}$ position of the chosen representative $\tilde{\omega}$. As the vectors $a_i$ will always have sum zero, $k(\omega)$ is well defined. Define also a function of formal variable $q$ $$\mu_N(a_1,\ldots,a_N) = \sum_{\omega \in \Omega_N} q^{\frac{k(\omega)}{2}}.$$

\subsubsection{Multiplicity of an unpointed vertex} Let $V$ be an unpointed trivalent vertex with outgoing vector weights the balanced set of vectors $(a_1,a_2,a_3)$. The multiplicity assigned to $V$ is $$m_V = (-i)\cdot [a_1\wedge a_2]_-.$$

\subsubsection{Multiplicity of a pointed vertex}\label{sec:multis} A vertex $V$ adjacent to edges with vector weights $a_1,\ldots,a_N$ defines a balanced set of vectors $\Delta_V = (a
_1,\ldots,a_N)$ in $\mathbb{R}^2$. From this data we define a function of $q$ $$m_V=\frac{1}{(N-1)!}\mu_N(a_1,\ldots,a_N),$$ which is the multiplicity of $V$. 

\subsubsection{Multiplicity of a tropical curve}
The multiplicity of the parametrised tropical curve ${h}:{\Gamma} \rightarrow \mathbb{R}$ is the product of $m_V$ over vertices of $\Gamma$.
\subsubsection{Blechman--Shustin multiplicity} 
The Blechman--Shustin multiplicity of a (pointed) vertex $V$ of valency $N$ with outgoing edges of vector weight $a_1,\ldots,a_N$ denoted $\theta_N(a_1,\ldots,a_N)$ is defined recursively on the valency. When $N=3$ they defined $$\theta_3(a_1,a_2,a_3) = [a_1 \wedge a_2]_+$$ and then recursively $$\theta_N(a_1,\ldots,a_N) = \sum_{i<j}\theta_{N-1}(a_1,\ldots,\hat{a}_i,\ldots,\hat{a}_j,\ldots,a_N, a_{i}+a_j)\theta_3(a_i,a_j,-(a_i+a_j)).$$ Here a hat denotes omission.

\subsubsection{Formula for Blechman--Shustin multiplicity} We provide an explicit formula for the Blechman--Shustin multiplicity.

\begin{proposition}\label{prop:Link To BS mult}
    For all choices of $a_i$, there is an equality $$\theta_N(a_1,\ldots,a_N) = \frac{N!}{3!}\mu_N(a_1,\ldots,a_N).$$
\end{proposition}

The proof of Proposition \ref{prop:Link To BS mult} was communicated to us by Thomas Blomme.

\begin{proof}
    The proof proceeds by induction on $N$. When $N = 3$ there are two cyclic orders which have representatives $\omega_1=(3,1,2)$ and $\omega_2=(3,2,1)$. We thus learn $$\mu_3(a_1, a_2, a_3)=q^{\frac{1}{2}a_1 \wedge a_2} + q^{-\frac{1}{2}a_1 \wedge a_2}=\theta_3(a_1, a_2, a_3).$$

    For the inductive step assume $\theta_{N-1}(a_1,\ldots,a_{N-1}) = \frac{(N-1)!}{3!}\mu_{N-1}(a_1,\ldots,a_{N-1})$. Let $W_N$ be the set of ordered pairs $(i,j)$ where $1 \leq i<j\leq N$. Define $\Omega$ to be the set of cyclic permutations of $(a_1,\ldots,a_N)$ and $\Omega_{i,j}$ the set of cyclic permutations of $$(a_1,\ldots,\hat{a}_i,\ldots,\hat{a}_j,\ldots,a_N,a_i+a_j)$$ where hat denotes omission. We now proceed with the following chain of equalities.
\begin{align*}
\theta_N (a_1, \ldots , a_N) &= \sum_{(i,j)\in W_N}\Big(q^{\frac{1}{2}a_i \wedge a_j} + q^{\frac{1}{2}a_j \wedge a_i}\Big)\theta_{N-1}(a_1,\ldots \hat{a}_i,\ldots \hat{a}_j, \ldots,a_i + a_j) \\
&= \frac{(N-1)!}{3!} \sum_{(i,j)\in W_N}\Big(q^{\frac{1}{2}a_i \wedge a_j} + q^{\frac{1}{2}a_j \wedge a_i}\Big)\mu_{N-1}(a_1,\ldots \hat{a}_i,\ldots \hat{a}_j,\ldots, a_i + a_j) \\
&= \frac{(N-1)!}{3!} \sum_{(i,j)\in W_N}\Big(q^{\frac{1}{2}a_i \wedge a_j} + q^{\frac{1}{2}a_j \wedge a_i}\Big)\sum_{\omega \in \Omega_{i,j}}q^{\frac{k(\omega)}{2}} \\
&= \frac{(N-1)!}{3!} \sum_{(i,j)\in W_N} \sum_{\omega \in \Omega_{i,j}}\Big(q^{\frac{k(\omega) + a_i \wedge a_j}{2}} + q^{\frac{k(\omega) - a_i \wedge a_j}{2}}\Big) \\
\end{align*}
In the sequel use $\Omega[i,j]$ for $i \ne j$ to denote the set of cyclic orders of $(a_1,\ldots,a_N)$ in which $a_i,a_j$ are adjacent. There is a two to one map $$\varphi:\Omega[i,j] \rightarrow \Omega_{i,j}$$ obtained by coupling $a_i$ and $a_j$. The map is two to one to account for the two orders of $a_i$ and $a_j$ and moreover $k(\varphi(x)) = k(x) \pm a_i \wedge a_j$ where the sign depends on the order of $a_i,a_j$. We deduce,
\begin{align*}\theta_n (a_1, \ldots , a_N)&= \frac{(N-1)!}{3!} \sum_{(i,j)\in W_N} \sum_{\omega \in \Omega[i,j]}q^{\frac{k(\omega)}{2}} \\
&= \frac{(N-1)!}{3!} \sum_{\omega \in \Omega}q^{\frac{k(\omega)}{2}} \mathsf{card}\{(i,j) : i<j, \omega \in \Omega[{i,j}]\}\\ 
&= \frac{N!}{3!} \sum_{\omega \in \Omega}q^{\frac{k(\omega)}{2}}\\ &= \frac{N!}{3!} \mu_N(a_1, \ldots, a_N).
\end{align*}
\end{proof}

\subsection{Moduli of tropical curves and maps} Following \cite{MR4145824}, note that assuming $n>0$ a parametrised tropical curve of degree $\Delta$ $$h: \Gamma \rightarrow \mathbb{R}^2$$ is specified by the following two data:

\begin{enumerate}[(1)]
\item the image of the first (necessarily contracted) unbounded edge $h(E_1)$;
\item the data of $\Gamma$ and the length of its edges.
\end{enumerate}
The slope of each unbounded edge is determined by the degree of $h$ and the slope of bounded edges are determined by the balancing condition.

For $m$ a non--negative integer, there is a cone complex $\mathcal{M}_{0,m}^\mathsf{trop}$ whose points biject with abstract tropical curves $\Gamma$ equipped with $m$ unbounded ends \cite{gathmann2008kontsevich,MaclaganSturmfels}. Datum (1) above is a point in $\mathbb{R}^2$ and the second datum is a point in $\mathcal{M}_{0,n+r}^\mathsf{trop}$.
In this way $\mathbb{R}^2\times \mathcal{M}_{0,n+r}^\mathsf{trop}$ is a moduli space of parametrised tropical curves \cite{ModStckTropCurve}.

There are $n$ evaluation maps 
$$\mathsf{ev}_i: \mathcal{M}_{0,n
+r}^\mathsf{trop}\times \mathbb{R}^2 \rightarrow \mathsf{Ev}^\mathsf{trop}  = \mathbb{R}^2$$ sending a parametrised tropical curve $h$ to the image of the $i^\mathsf{th}$ unbounded edge under $h$. This image is necessarily a single point. The cone complex $\mathcal{M}_{0,n
+r}^\mathsf{trop}$ has a natural embedding into a vector space such that the evaluation maps are pulled back from linear functions \cite{gathmann2008kontsevich}.

\subsubsection{Dimension of a cone} The \textit{combinatorial type} of a tropical curve corresponding to a point $p$ of $\mathbb{R}^2 \times \mathcal{M}_{0,n+r}^\mathsf{trop}$ is the data of the corresponding cone $\sigma_p$ of $\mathcal{M}_{0,n+r}^\mathsf{trop}$. Consider now a parametrised tropical curve of degree $\Delta$. Assume our tropical curve has valency $3$ at all vertices except the vertices supporting one of the first $n$ marked points $\{q_1,\ldots,q_n\}$ which are $\ell_i +2$ valent for $\ell_i$ non--negative integers. 

\begin{lemma}\label{lem:DimOfCone}
    There is an equality $$\mathsf{dim}(\sigma_p) = r-3 - \sum_i(\ell_i-3)$$
\end{lemma}

\begin{proof}
    The dimension of a cone in $\mathcal{M}_{0,n}^\mathsf{trop}$ is the number of bounded edges in a tropical curve of the corresponding combinatorial type. We induct on the number of vertices. 

    If there is a single vertex then there are no interior edges, $r$ coincides with the valency and the equality reads $$0 = r-3 -(r-3),$$ so the result is true.

    For the inductive step suppose we add a vertex $V_0$ of valency $k_0$ to the graph by replacing an unbounded edge with a bounded one to $V_0$. This increases $\mathsf{dim}(\sigma_p)$ by one because there is one new unbounded edge. The value of $r$ is increased by $k_0 - 2$ and so the right hand side of our equality increases by one overall.
\end{proof}

\subsection{Tropical counting problem} Recall we have fixed a pair $(\Delta^\circ,\textbf{k})$. In this section we associate a function of $q$ to this data. The function of $q$ is a count of tropical curves with multiplicity passing through a generic tuple of points. 

\subsubsection{Passing through points} Let $p=(p_1,\ldots,p_k)$ be a tuple of points with $p_i \in \mathbb{R}^2$. A parametrised tropical curve ${h}$ of degree $\Delta$ is said to \textit{pass through $p$ with degree $\textbf{k}$} if ${h}(E_i) = p_i$ for $i=1,\ldots,n$ and $E_i$ is attached to a vertex of valency at least $k_i+2$. We say a parametrised tropical curve ${h}$ through $p$ of degree $\Delta$ is \textit{rigid} if ${h}$ is unique among its combinatorial type in having this property.

\subsubsection{Weighted count of tropical curves} For $p$ a generic tuple of points in $\mathbb{R}^2$, let $T_{\Delta,p}^{\textbf{k}}$ be the set of rigid parametrised tropical curves  of degree $\Delta$ passing through $p$ with degree $\textbf{k}$.

\begin{proposition}\label{prop:TropCurveCount}
There is an open dense subset $U_{n}^{\textbf{k} }$ of $\mathsf{Ev}^\mathsf{trop}$ such that whenever $p\in U_{n}^{\textbf{k} }$ then $T_{\Delta,p}^{\textbf{k}}$ is a finite set and the valency of the vertex supporting unbounded edge $E_i$ is $k_i + 2$. Moreover we may choose $U_{n}^{\textbf{k} }$ such that all parametrised tropical curves passing through $p$ with degree $\textbf{k}$ are rigid.
\end{proposition}
In the sequel we assume $p_i \neq p_j$ whenever $i$ and $j$ are distinct without further comment.

\begin{remark}
Since there are only finitely many combinatorial types of rigid parametrised tropical curves of degree $\Delta$, it is automatic that the set $T_{\Delta,p}^{\textbf{k}}$ is finite.
\end{remark} 

\begin{proof}
    Let $T$ be the set of points in $\mathcal{M}_{0,n}\times \mathbb{R}^2$ corresponding to degree $\Delta$ tropical curves passing through $p$ with degree $\textbf{k}$. Our task is to show that by choosing $U_n^{\textbf{k}}$ generically we may ensure first $T =T_{\Delta,p}^{\textbf{k}} $ and second each point of $T$ lies in the interior of cones in which the vertex mapped to marked point $p_i$ has valency precisely $k_i+2$. 
    
    Since $\mathcal{M}_{0,n}^\mathsf{trop}$ has finitely many cones, it suffices to identify for each cone $\sigma$ a dense open $U_\sigma$ of $\mathsf{Ev}^\mathsf{trop}$ such that $T\cap \sigma = T_{\Delta,p}^{\textbf{k}}\cap \sigma$ and if this set is not empty then the combinatorial type of tropical curves corresponding to points in the interior of $\sigma$ have valency $k_i+2$ at the vertex supporting $E_i$. 
    
    We assume marked point $q_i$ on $\Gamma$ has valency $k_i'+2$ for every tropical curve associated to a point in the interior of $\sigma$. If a point $p$ of $\mathcal{M}_{0,n} \times \mathbb{R}^2$ corresponds to a tropical curve in $T_{\Delta,p}^{\textbf{k}}$ then the vertex carrying marked point $i$ must have valency at least $k_i$. Thus we may assume the vertex of $\Gamma$ mapped to $p_i$ has valency at least $k_i$. Assuming the set $\{p_i\}$ of marked points are distinct, Lemma~\ref{lem:DimOfCone} implies $\mathsf{dim}(\sigma)\leq N-3-\sum_i (k_i-3)$. Thinking of $\sigma$ as a cone embedded in $\mathbb{R}^\ell$ , the restriction of $\mathsf{Ev}$ to $\sigma$ then specifies a linear map $\mathbb{R}^{\mathsf{dim}(\sigma)} \rightarrow \mathbb{R}^{2n}$. For a generic choice of $\{p_i\}$ and for fixed $\sigma$ whenever $\sigma$ contains a point of $T$ this linear map surjects. Rank-nullity gives a lower bound and the fact $k_i'\geq k_i$ for all $i$ gives the following upper bound $$N-3-\sum_i (k_i-3) \geq N-3-\sum_i (k_i'-3) \geq  N-3-\sum_i (k_i-3).$$  from which we deduce $k_i' = k_i$ for all $i$.
\end{proof}

\begin{definition}\label{defn:TropCurveCount}
    Recall notation $m_V$ for multiplicities of vertex $V$ defined in Section \ref{sec:multis}. Fix $ p = (p_1,\ldots,p_n)$ in $U_n^{\textbf{k}}$ and define $$\Ntrop(q) = \sum_{h \in T_{\Delta,p}^{\textbf{k}}} \prod_{V \in V(\Gamma)}m_V(q).$$ Define also $\Ntrop(1)=\Ntrop$.
\end{definition}
A priori the count $\Ntrop(q)$ depends on the choice of $p$. We suppress this dependence from our notation as it is independent a posteriori.

\subsection{Anatomy of tropical curves} We record properties of tropical curves which will be of later use.

\begin{proposition}\label{prop : orienatation}
    The complement of the pointed vertices of a parametrised tropical curve $h$ in $T_{\Delta,p}^{\textbf{k}}$ is a union of trees, each with a single unbounded edge.
\end{proposition}
One may specify a component of this compliment by the collection $\kappa_i$ of vertices which lie in its closure. For each $\kappa_i$ define a parametrised tropical curve $$h_i:\Gamma_i\rightarrow \mathbb{R}^2$$ as follows. First define $\Gamma_i$ from $h$ by deleting from $\Gamma$ every vertex not in $\kappa_i$, and also deleting every edge which has at least one end not in $\kappa_i$. The map $h_i$ is the restriction of $h$ to $\Gamma_i$.
\begin{proof}
     All $\Gamma_i$ are trees by \cite[Proposition 4.19]{Milk}. We are required to show $\Gamma_i$ has precisely one unbounded edge which is not contracted by $h_i$. This is because unbounded edges contracted by $h_i$ are the pointed vertices which are deleted in the proposition statement.

    Assume that $\Gamma_i$ has $k$ contracted unbounded edges and $\ell$ unbounded edges which are not contracted. The moduli space of parametrised tropical curves with the same combinatorial type as $\Gamma_i$ is written $\sigma_i$ and has dimension $2k + \ell -1$. Since $h$ was rigid, so is $h_i$. In order for $h_i$ to be rigid, the dimension of $\sigma_i$ must equal two times the number of bounded edges. That is, $$2k = 2k + \ell -1.$$ Thus there is precisely one unbounded edge which is not contracted.
\end{proof}

\begin{figure}[h]
    \centering
    \includegraphics[width=12cm]{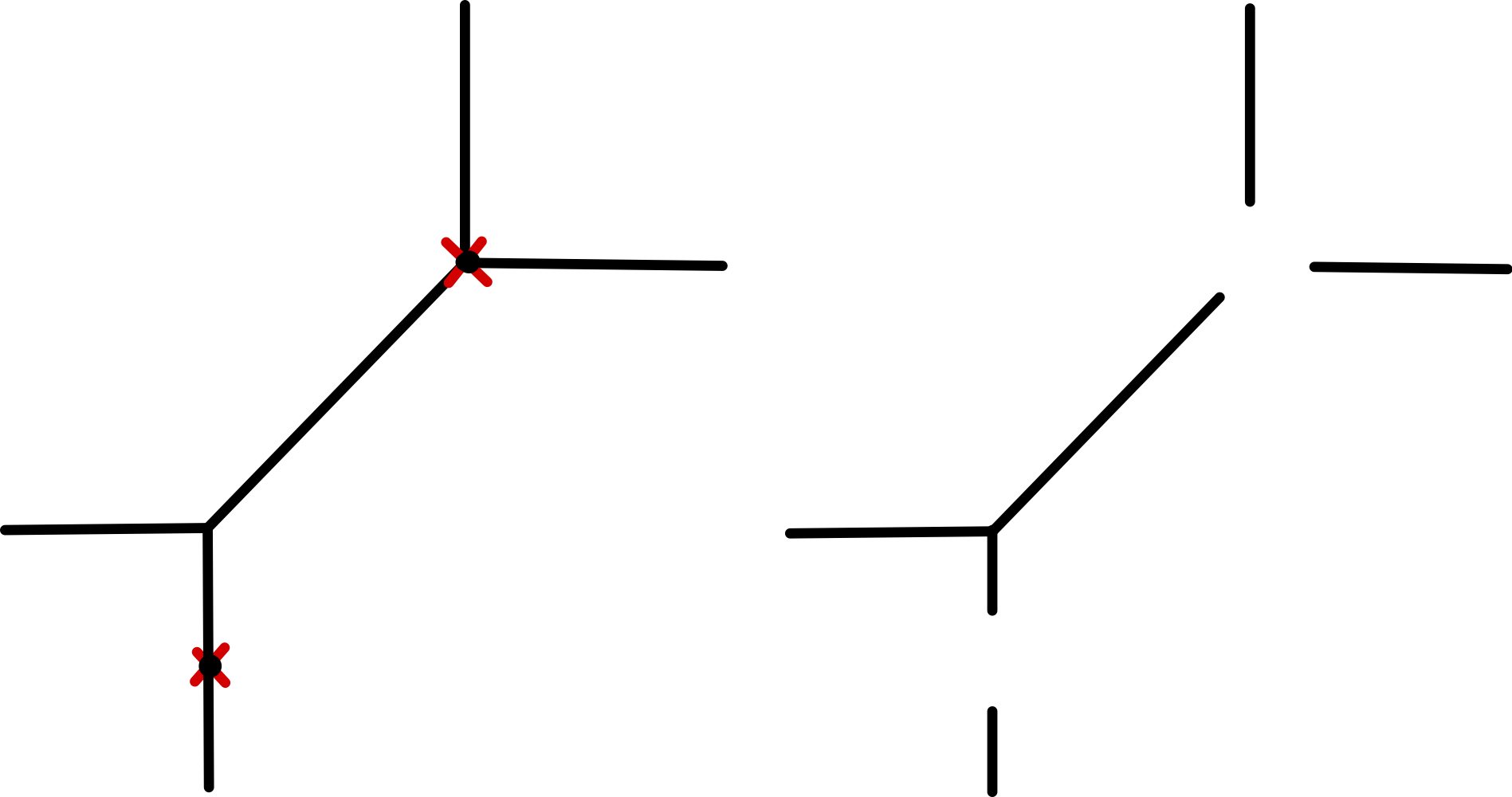}
    \caption{Left a curve in $T_{\Delta,p}^{(1,0)}$ for $p$ the two red crosses shown and $\Delta^\circ = \{(1,0), (-1,0), (0,1),(0,-1)\}$. Right, the union of trees with a single unbounded edges discussed in Proposition \ref{prop : orienatation}.}
\end{figure}


\section{Precise statement of main theorem}\label{Section : Statement + LogGW}
Fix $(\Delta^\circ,\textbf{k})$ as in Section \ref{Section : Tropical} and recall notation $\Delta$ for the associated matrix from the same section. Fixing a lattice direction $\rho$ write $n_\rho$ for the sum of the lattice lengths of the vectors in $\Delta^\circ = \{\delta_1,\ldots,\delta_r\}$ of direction $\rho$. 
Associated to $(\Delta^\circ,\textbf{k})$ is the following data. 
\begin{enumerate}[(1)]
    \item Set $X_\Delta$ the unique toric surface corresponding under the toric dictionary to the fan with rays in the direction of the vectors of $\Delta^\circ$. See \cite{Fulton+2016} for the toric dictionary. 
    \item Set $\beta_\Delta$ the unique curve class on $X_{\Delta}$ whose intersection with the boundary divisor corresponding to the ray $\rho$ is $n_\rho$. 
\end{enumerate}  
We consider $X_\Delta$ as a logarithmic scheme with the divisorial logarithmic structure from its toric boundary.

\subsection{Logarithmic Gromov--Witten Invariants}\label{sec:logGW}
The moduli space parametrising $(r+n)$-pointed, genus $g$ stable maps to $X_{\Delta}$ of degree $\beta_{\Delta}$ with the tangency to the toric boundary in the final $r$ markings given by the vectors $\delta_i$, $i=1,\dots r$ is not proper. The space of \emph{stable logarithmic maps of type $\Delta$} written $\mathsf{M}_{g,\Delta} = \mathsf{M}_{g,\Delta}^{\log}(X_{\Delta}|\partial X_{\Delta},\beta_{\Delta})$ is a compactification \cite{chen2011stable,abramovich2011stable, gross2012logarithmic}. In the relative situation, such a moduli space exists for any logarithmically smooth proper morphism $X \rightarrow S$.

\subsubsection{Moduli of curves}\label{sec:ModuliOfCurves} Writing $\overline{\mathcal{M}}_{g,m}$ for the moduli space of stable genus $g$ curves with $m$ marked points. This space comes equipped with universal curve $p:\mathcal{C} \rightarrow \overline{\mathcal{M}}_{g,m}$. Since under the hypotheses of the introduction $m=n+r>2$, there is a forgetful morphism $$\pi:\mathsf{M}_{g,\Delta} \rightarrow \overline{\mathcal{M}}_{g,n+r}.$$ The moduli space of stable curves carries two flavours of tautological bundle of import to us. \begin{itemize}
\item The \textit{Hodge bundle} $\mathbb{E}_g = p_\star \omega_p$ where $\omega_p$ is the relative dualising sheaf of $p$. We write $\lambda_g = c_g(\mathbb{E}_g)$.
\item Note $\overline{\mathcal{M}}_{g,n+r}$ carries $n+r$ tautological sections identifying the marked points. Denote the first $n$ sections as $s_1,\ldots,s_n$ and define $$\psi_i = c_{1}(s_i^\star \omega_p).$$ 
\end{itemize} Both $\psi_i$ and $\lambda_g$ can be pulled back along $\pi$ to define tautological classes on the moduli space of stable logarithmic maps. 

\begin{remark} The $\psi_i$ and $\lambda_g$ classes on the moduli space of logarithmic stable maps can instead be defined directly as Chern classes of tautological bundles on $\mathsf{M}_{g,\Delta}$. These definitions are equivalent: see \cite[Proposition 3.4]{MandelRuddat} for the case of $\psi_i$ classes. 

For the $\lambda_g$ class, we will argue that the Hodge bundle on $\mathsf{M}_{g,\Delta}$ coincides with the pullback of the Hodge bundle on $\overline{\cM}_{g,n+r}$ along $\pi$. Indeed write $ q:C\rightarrow \mathsf{M}_{g,\Delta}$ for the domain universal curve and $p'=\pi^\star p: \pi^\star\mathcal{C}\rightarrow \mathsf{M}_{g,\Delta}$ for the pullback of the universal curve from the moduli of stable curves. There is a stabilisation map $\mathsf{stab}:C \rightarrow \pi^\star\mathcal{C}$ over $\mathsf{M}_{g,\Delta}$. In this notation, the promised identification of Hodge bundles is an isomorphism between $R^1p_\star' \mathcal{O}_{\pi^\star\mathcal{C}}$ and $R^1q_\star \mathcal{O}_{C}$. The identification is immediate from two facts: first the five term exact sequence associated to the relative Leray spectral sequence for the functors $\mathsf{stab}_\star$ and $p'_\star$; second properties of stabilisation for a family of prestable curves \cite[TAG 0E8A]{stacks-project}.  
\end{remark}

\subsubsection{Evaluation maps} For each of the $n+r$ marked points $\{q\}$ 
 there is a tautological morphism $$\mathsf{ev}_q: \mathsf{M}_{g,\Delta}\rightarrow X_\Delta$$ called the \textit{evaluation morphism} associated to $q$. This morphism sends a stable logarithmic map to the image of $q$ in $X_\Delta$. Write $\mathsf{ev}_1,\ldots,\mathsf{ev}_n$ for evaluation maps at the first $n$ sections. 

\subsubsection{Invariants}
The moduli space $\mathsf{M}_{g,\Delta}$ carries a virtual fundamental class $\left[\mathsf{M}_{g,\Delta}\right]^\mathsf{vir}$ allowing us to define \textit{logarithmic Gromov--Witten invariants}. We will consider the following descendant logarithmic Gromov--Witten invariants with a $\lambda_g$ insertion. $$N^\textbf{{k}}_{g,\Delta} = \int_{\left[\mathsf{M}_{g,\Delta}\right]^\mathsf{vir}}(-1)^g\lambda_g\prod_i\psi_i^{k_i}\mathsf{ev}_i^\star (\mathsf{pt}).$$
\subsection{Main theorem}\label{sec:MainTheorem} We are now ready to state our main theorem.

\begin{theorem}[Theorem \ref{thm:main}]
    Set $q=e^{iu}$. There is an equality $$ \sum_{g\geq 0}N_{g,\Delta}^{\textbf{k}} u^\mathsf{2g+ r-2-\sum_ik_i}= \Ntrop(q).$$
\end{theorem}

\begin{remark}
    Bousseau \cite{bousseau2018tropical} proved a tropical correspondence result for integrals of the form $$\int_{\left[\mathsf{M}_{g,\Delta}\right]^\mathsf{vir}}(-1)^{g-g_{\Delta}}\lambda_{g-g_\Delta}\prod_i\mathsf{ev}_i^\star (\mathsf{pt}).$$ We provide a new proof of the special case $g_\Delta=0$ of Bousseau's result: set $\textbf{k}$ to be the zero vector in Theorem \ref{thm:main}. Both our proof and Bousseau's work proceed by reducing to computing \textit{vertex contributions.} Bousseau computes vertex contributions through a consistency argument. We are able to handle descendants because our computation of the vertex contributions in Sections \ref{Section : LogGW and DR} and \ref{Section : Integrable Hierachies} involves intersection theory on the moduli space of stable curves. Furthermore, in this case, the vertex contributions are of the same form for $g_{\Delta} \neq 0$, so the same technique can be used to reprove Bousseau's general result.
\end{remark}

\section{Decomposition and gluing}\label{Section : Decomposition and Gluing}

Fix once and for all a generic choice of points $p$ in $\mathbb{R}^2$. By generic we mean $p\in U^\textbf{k}_n$ from Proposition \ref{prop:TropCurveCount}. The arguments of this section follow \cite{bousseau2018tropical} and proceed in the following steps. 
\begin{enumerate}[(1)]
    \item Use the tropical curves in $T_{\Delta,p}^{\textbf{k}}$ to build a toric degeneration of $X_{\Delta}$, see Section \ref{sec:toric degen}.
    \item {Appeal to the decomposition formula of \cite{abramovich2020decomposition} to turn a computation on the central fibre of this degeneration into a sum over the tropical curves, see Section \ref{sec:decomposition}.}
    \item Use the gluing theorem of \cite{ranganthan2022logarithmic} 
    to decompose the contribution of each tropical curve as a product of contributions from each vertex. See Section \ref{sec:Gluing}.
\end{enumerate}  

\subsection{Toric degeneration}\label{sec:toric degen} 
Following ~\cite[4.2]{bousseau2018tropical},~\cite[Proposition 3.9]{NishSieb06} and ~\cite[Lemma 3.1]{MandelRuddat}, the set of tropical curves $T_{\Delta,p}^{\bf{k}}$ determine a polyhedral decomposition 
$\mathcal{P}$ of $\RR^2$ satisfying 
\begin{itemize}
    \item the asymptotic fan of $\cP$ is the fan of $X_{\Delta}$.
    \item The image of the vertices of any $h: \Gamma \rightarrow \RR^2$ in $T^{\bf{k}}_{\Delta,p}$ are vertices of $\cP$ and the image of any edges of $h$ are a union of edges of $\cP$.
\end{itemize}
Moreover by rescaling $\RR^2$, we can assume $\cP$ to be an integral polyhedral decomposition. 
This determines a degeneration $$\nu: X_{\mathcal{P}} \rightarrow \AAA^1$$ with general fibre $X_{\Delta}$ and special fibre $X_0 = \cup_{V} X_V$ a union of components indexed by vertices of the decomposition $\mathcal{P}$. Since $\nu$ is toric, equipping $X_{\mathcal{P}}$ and $\AAA^1$ with the divisorial logarithmic structures from their respective toric boundaries, makes $\nu$ into a logarithmically smooth morphism. 

Restricting to the central fibre, there is a logarithmically smooth morphism $X_0\rightarrow \pt_\NN$. Write $\mathsf{M}_{g,\Delta}(X_0/\pt_{\NN})$ for the moduli space of stable logarithmic maps, where a family over a fine and saturated logarithmic scheme $S$ is a commutative diagram
\begin{equation*}
    \begin{tikzcd}
C \arrow[d] \arrow[r] & X_0 \arrow[d] \\
S \arrow[r]           & \pt_{\NN}    
\end{tikzcd}
\end{equation*}
with $C/S$ a logarithmic curve, and the contact order data is specified by $\Delta$. The notation highlights that although the discrete data is unchanged, the target has been degenerated.}

Each of the $p_j$ determines a section of $\nu$ ~\cite[Section 4.2]{bousseau2018tropical}. The restriction of this section to the special fibre defines a point $$i_{P^0} : (P^0_1,\dots,P^0_{n}) \hookrightarrow X_0^n.$$ Define $\mathsf{M}_{g,\Delta}(X_0/\pt_{\NN},P^0)$ as the fibre product
\begin{equation*}
 \begin{tikzcd}
\mathsf{M}_{g,\Delta}(X_0/\pt_{\NN},P^0) \arrow[d] \arrow[r]                      & \mathsf{M}_{g,\Delta}(X_0/\pt_{\NN}) \arrow[d] \\
{P^0= (P^0_1,\dots,P^0_n)} \arrow[r, "i_{P^0}", hook] & (X_0)^n.            
\end{tikzcd}
\end{equation*}
By deformation invariance of logarithmic Gromov--Witten invariants and ~\cite[Example 6.3.4 (a)]{Fulton}, 
\begin{equation*}
   N_{g,\Delta}^{\textbf{k}} = \int_{[\mathsf{M}_{g,\Delta}(X_0/\pt_{\NN},P^0)]^{\vir}} (-1)^g \lambda_g \prod_{i=1}^n \psi_i^{k_i}
\end{equation*}
where $[\mathsf{M}_{g,\Delta}(X_0/\pt_{\NN},P^0)]^\mathsf{vir} = i_{P^0}^{!}[\mathsf{M}_{g,\Delta}(X_0/\pt_{\NN})]^\mathsf{vir}$.
\subsection{Decomposition}\label{sec:decomposition}

We now use the decomposition formula of \cite{abramovich2020decomposition} to write $N_{g,\Delta}^{\textbf{k}}$ of integrals indexed by the tropical curves in $T^{\textbf{k}}_{\Delta,p}$. 

\subsubsection{Genus $g$ from genus 0} All tropical curves in $\SetTropCurve$ have genus zero, however the decomposition formula provides the integral as a sum over genus $g$ tropical curves. Given $h:\Gamma \rightarrow \RR^2 \in \SetTropCurve$ we can build genus $g$ tropical curves $\tilde{h} : \tilde{\Gamma} \rightarrow \RR^2$ in two steps. \begin{enumerate}[(1)]\item Add a genus zero bivalent unpointed vertex to $\Gamma$ at each point $h^{-1}(V)$ for $V$ a vertex of $\cP$. 
\item Distribute an assignment of genus $g_V \in \NN$ to the vertices $V$ of $\tilde{\Gamma}$ such that $\sum_{V \in \tilde{\Gamma}} g_V = g$. \end{enumerate} We call the resulting set of parametrised tropical curves $T^{g,\textbf{k}}_{\Delta,p}$. 

\subsubsection{Maps marked by $\tilde{h}$} Not all rigid parametrised tropical curves lie in $T^{g,\textbf{k}}_{\Delta,p}$. For \emph{any} rigid, genus $g$ parametrised tropical curve $\tilde{h}: \tilde{\Gamma} \rightarrow \RR^2$ passing through $p$, an $n$-pointed, genus $g$ stable logarithmic map \emph{marked by $\tilde{h}$} is the following data.

\begin{enumerate}[(1)]
    \item An $n$-pointed genus $g$ stable logarithmic map $f : C/\pt_{\cM} \rightarrow X_0/\pt_{\NN}$ of
type $\Delta$ passing through $P^0$.
\item For each vertex $V \in V(\tilde{\Gamma})$, an ordinary stable map $f_V:C_V \rightarrow X_{\Delta_V}$ of class $\beta_{\Delta_V}$ with marked points $x_\delta$ for any $\delta \in \Delta_V$ such that $f_V(x_{\delta}) \in D_\delta$, where $D_\delta$ is the toric divisor dual to $\delta$.
\end{enumerate}
These are subject to the following requirements. The underlying curve of $C$ is isomorphic to the curve given by gluing the curves $C_V$ along the points $x_{\delta}$ according to $\tilde{\Gamma}$. Moreover, under the isomorphism above, the scheme-theoretic morphism $C \rightarrow X_0$ obtained by gluing the morphisms $f_V$ agrees with the underlying morphism of $f$.

The moduli space of $n$-pointed genus $g$ stable logarithmic maps marked by $\tilde{h}$, denoted $\mathsf{M}^{\tilde{h}}_{g,\Delta}$ is a proper Deligne--Mumford stack with a natural perfect obstruction theory and a forgetful morphism $$i_{\tilde{h}} : \mathsf{M}^{\tilde{h}}_{g,\Delta} \rightarrow \mathsf{M}_{g,\Delta}(X_0/\pt_{\NN},P^0).$$
\subsubsection{From maps marked by $\tilde{h}$ to Gromov--Witten invariants}
     For each $\tilde{h} \in T^{g,\textbf{k}}_{\Delta,p}$  define
\begin{equation*}
    N_{g,\Delta}^{\tilde{h},\textbf{k}}= \int_{[\mathsf{M}^{\tilde{h}}_{g,\Delta}]^\mathsf{vir}} (-1)^g \lambda_g \prod_{i=1}^n \psi_i^{k_i}
\end{equation*}
where we abuse notation by writing $\lambda_g = i_{\tilde{h}}^\star\lambda_g$ and $\psi_i = i_{\tilde{h}}^\star\psi_i$. The decomposition formula of \cite{abramovich2020decomposition} tells us that $$N_{g,\Delta}^{\textbf{k}} = \sum_{\tilde{h}} \frac{n_{\tilde{h}}}{|\mathsf{Aut}(\tilde{h})|}N_{g,\Delta}^{\tilde{h},\textbf{k}}.$$  Here $n_{\tilde{h}}$ denotes the smallest positive integer such that $\tilde{h}$ has integral vertices after being rescaled by a factor of $n_{\tilde{h}}$. The number $|\mathsf{Aut}(\tilde{h})|$ is the order of the group of automorphisms of the parametrised tropical curve $\tilde{h}$.  

\begin{proposition}\label{prop:decomp} There is an equality of rational numbers
    \begin{equation*}
     N_{g,\Delta}^{\textbf{k}}  = \sum_{\tilde{h} \in T_{\Delta,p}^{g,\textbf{k}}}  N_{g,\Delta}^{\tilde{h},\mathbf{k}}.
\end{equation*}
\end{proposition}
\begin{proof}
    Any $\tilde{h} \in T^{g,\textbf{k}}_{\Delta,p}$ is a rigid parametrised tropical curve. Since $\cP$ is chosen to be integral all such curves have $n_{\tilde{h}} = 1$. Only tropical curves in $T^{g,\textbf{k}}_{\Delta,p}$ contribute because all other rigid tropical curves have positive Betti number. The presence of the $\lambda_g$ insertion ensures that such tropical curves do not contribute, see ~\cite[Lemma 8]{bousseau2018tropical}. Finally, there are no automorphisms of $\tilde{h}$ because it is rational.
\end{proof}
\subsection{Vertex contribution}\label{VertexContributionSub} Let $\tilde{h}: \tilde{\Gamma} \rightarrow \RR^2$ be an element of $T^{g,\textbf{k}}_{\Delta,p}$, there are four types of vertices of $\tilde{\Gamma}$.
\begin{enumerate}[(1)]
    \item Trivalent unpointed vertices, coming from $\Gamma$, we call the set of these $V^{3}(\tilde{\Gamma})$.
    \item Bivalent pointed vertices, coming from $\Gamma$, we call the set of these $V^{2\mathsf{p}}(\tilde{\Gamma})$.
        \item Bivalent unpointed vertices, not coming from $\Gamma$, we call the set of these $V^{2}(\tilde{\Gamma})$.
            \item $m$-valent ($m \geq 3$) pointed vertices, coming from $\Gamma$, we call the set of these $V^{\mathsf{mp}}(\tilde{\Gamma})$.
\end{enumerate}

Recall $\Delta_V^\circ$ denotes the balanced multiset of vectors arising from edges adjacent to $V$. Write $\Delta_V$ for the matrix whose columns are the vectors $v_{V,E}$ for $E$ \emph{any} edge adjacent to $V$. Just as in Section \ref{Section : Statement + LogGW}, this determines a toric surface $X_{\Delta_V}$, a curve class $\beta_{\Delta_V}$ and tangency conditions for $\beta_{\Delta_V}$ with respect to the toric boundary. If the elements of $\Delta_V^{\circ}$ do not span $\RR^2$, replace $X_{\Delta_V}$ with some toric compactification, the choice is unimportant.

Recall from Proposition \ref{prop : orienatation} that the complement of the pointed vertices are trees with a single unbounded edge. Consequently we may choose a consistent orientation from pointed vertices to the unbounded edges. We now fix this orientation without further comment.

\subsubsection{Trivalent unpointed contribution}\label{def : trivalent contribution}

Let $V$ be an unpointed trivalent vertex of $\tilde{\Gamma}$. For our fixed orientation $X_{\Delta_V}$ has two divisors oriented inwards $D_1,D_2$ with associated edges $E_{V}^{\mathsf{in},1},E_{V}^{\mathsf{in},2}$ and one oriented outwards $D_{\mathsf{out}}$. Let $\mathsf{M}_{g_V,\Delta_V}$ be the moduli space of stable logarithmic maps to $X_{\Delta_V}$ of genus $g_V$ and of type $\Delta_V$. We have evaluation morphisms with image in the toric boundary of $X$. We can therefore think of these evaluation maps as morphisms
$$(\ev_{\mathsf{in},1},\ev_{\mathsf{in},2}, \ev_{\mathsf{out}}) : \mathsf{M}_{g_V,\Delta_V} \rightarrow D_{1} \times D_2 \times D_{\mathsf{out}}.$$
Define the \textit{trivalent unpointed contribution}

\begin{equation*}
    N_{g_V,V} = \int_{[\mathsf{M}_{g_V,\Delta_V}]^\mathsf{vir}} (-1)^{g_V}\lambda_{g_V}\ev_{\mathsf{in},1}^\star (\pt_{D_1})\ev_{\mathsf{in},2}^\star (\pt_{D_2}).
\end{equation*}

\subsubsection{Bivalent pointed contribution}

Let $V$ be a pointed bivalent vertex of $\tilde{\Gamma}$. Let $\mathsf{M}_{g_V,\Delta_V}$ be the moduli space of stable log maps to $X_{\Delta_V}$ of genus $g_V$ and of type $\Delta_V$. We have an evaluation morphism $\ev_V : \mathsf{M}_{g_V,\Delta_V} \rightarrow X_{\Delta_V}$. Define the \textit{bivalent pointed contribution}

\begin{equation*}
    N_{g_V,V} = \int_{[\mathsf{M}_{g_V,\Delta_V}]^\mathsf{vir}} (-1)^{g_V}\lambda_{g_V}\ev_V^\star (\pt).
\end{equation*}

\subsubsection{Bivalent unpointed contribution}

Let $V$ be a bivalent unpointed vertex of $\tilde{\Gamma}$. Let $\mathsf{M}_{g_V,\Delta_V}$ be the moduli space of stable log maps to $X_{\Delta_V}$ of genus $g_V$ and of type $\Delta_V$. The orientation defines a divisor $D_{\mathsf{in}}$ with associated edge $E_{\mathsf{in}}$. Define the \textit{bivalent unpointed contribution}

\begin{equation*}
    N_{g_V,V} = \int_{[\mathsf{M}_{g_V,\Delta_V}]^\mathsf{vir}} (-1)^{g_V}\lambda_{g_V}\ev_{\mathsf{in}}^\star (\pt_{D_{\mathsf{in}}}).
\end{equation*}

\subsubsection{ $m$-Valent pointed contribution}\label{def : higher valency contribution}

Let $V$ be an $m$-valent pointed vertex of $\tilde{\Gamma}$ ($m \geq 3$). Since it is a higher valency pointed vertex there is a corresponding $k_i = m-2 \geq 1$. Let $\mathsf{M}_{g_V,\Delta_V}$ be the moduli space of stable log maps to $X_{\Delta_V}$ of genus $g_V$ and of type $\Delta_V$ . We have an evaluation morphism $\ev_V : \mathsf{M}_{g_V,\Delta_V} \rightarrow X_V$. Define the \textit{$m$-valent pointed contribution}

\begin{equation*}
    N_{g_V,V} = \int_{[\mathsf{M}_{g_V,\Delta_V}]^\mathsf{vir}} (-1)^{g_V}\lambda_{g_V} \ev_V^\star (\pt) \psi^{m-2}.
\end{equation*}

\subsection{Gluing the vertices}\label{sec:Gluing}

Proposition \ref{prop:decomp} reduces computing $N^\textbf{k}_{g,\Delta}$ to integrals over (the virtual class of) $\mathsf{M}^{\tilde{h}}_{g,\Delta}$. We now express these integrals as a product of the vertex contributions defined in Section \ref{VertexContributionSub}. The intuitive picture is that curves mapping to $X_{\Delta_V}$ glue together to form a map from a curve to $X_0$. In \cite{bousseau2018tropical}, the author proves a gluing statement at the level of virtual classes on the locus where the curve does not map into the torus fixed points of any $X_{\Delta_V}$. It is then shown that the $\lambda_g$ insertion will kill any contribution supported away from this locus. This is no longer true by the same arguments in this situation due to the presence of higher valent vertices; the key being that ~\cite[Lemma 17]{bousseau2018tropical} is false in our situation. Instead, we use the gluing formula of \cite{ranganthan2022logarithmic}. In general, this gluing formula is somewhat different due to allowing target expansions, but in the special case of degeneration of surfaces with at worst triple points, then the naive gluing formula holds when cutting along a single edge ~\cite[6.5.3]{ranganthan2022logarithmic}, and we will use this to prove that $N_{g,\Delta}^{\textbf{k}}$ is equal to a product of the vertex contributions.

\subsubsection{Notation}

For $V$ a pointed vertex of $\tilde{\Gamma}$, there is an $i_V \in \{1,\dots,n\}$ such that $P_{i_V}^0 \in X_{\Delta_V} \subset X_0$. Let $\mathsf{M}'_{g_V,\Delta_V}$ be the fibre product
\begin{equation*}
    \begin{tikzcd}
{\mathsf{M}'_{g_V,\Delta_V}} \arrow[d] \arrow[r] & {\mathsf{M}_{g_V,\Delta_V}} \arrow[d, "\ev_V"] \\
P_{i_V}^0 \arrow[r, hook]                            & X_{\Delta_V}                             
\end{tikzcd}
\end{equation*}
As in Section \ref{sec:toric degen}, there is a virtual class $[\mathsf{M}'_{g_V,\Delta_V}]^{\vir}$ on $\mathsf{M}'_{g_V,\Delta_V}$ given by Gysin pullback which pushes forward to $\ev_V^*({\pt}) \cap [\mathsf{M}_{g_V,\Delta_V}]^{\vir}$. 
For $V$ an unpointed vertex, set $\mathsf{M}'_{g_V,\Delta_V} = \mathsf{M}_{g_V,\Delta_V}$

Observe that for each $E \in E_f(\tilde{\Gamma})$, there are two natural maps $$\prod_{V \in V(\tilde{\Gamma})} \mathsf{M}'_{g_V,\Delta_V} \rightarrow D_E$$ because each edge is adjacent to two vertices. With this observation we build a map $$\ev_{(e)} : \prod_{V \in V(\tilde{\Gamma})} \mathsf{M}'_{g_V,\Delta_V} \rightarrow \prod_{E \in E_{f}(\tilde{\Gamma})} D_E^2.$$ 

For $\tilde{h}\in T^{g,\bf{k}}_{\Delta, p}$ we define the map:
$$\mathsf{cut}: \mathsf{M}_{g, \Delta}^{\tilde{h}} \rightarrow \prod_{V\in V(\tilde{\Gamma})} \mathsf{M}'_{g_V, \Delta_V},$$
by partially normalizing the source curve of a stable logarithmic map marked by $\tilde{h}$ (See ~\cite[p.$36-37$]{bousseau2018tropical} for a detailed construction).
Let $\kappa_{\tilde{h}}: \prod_{E \in E_{f}(\tilde{\Gamma}) } D_E \rightarrow \prod_{E \in E_{f}(\tilde{\Gamma})} D^2_E$ be the diagonal embedding and similarly, for an edge $F \in E_f(\tilde{\Gamma})$, let $\kappa_F :D_F \rightarrow D_{F} \times D_F$ the diagonal embedding. We will use the same notation to refer to the cohomology class associated to the diagonal embedding.
\begin{proposition}\label{prop:EqualityChow} There is an equality of Chow cycles
$$\mathsf{cut}_{\star}([\mathsf{M}_{g, \Delta}^{\tilde{h}}]^{\vir})= \left(\prod_{E \in E_{f}(\tilde{\Gamma})} w(E) \right) \left(\mathsf{ev}_{(e)}^{\star}(\kappa_{\tilde{h}})\cap \prod_{V \in V(\tilde{\Gamma})} [\mathsf{M}'_{g_V,\Delta_V}]^{\vir}\right).$$
\end{proposition}
\begin{proof}
We induct on $| V(\tilde{\Gamma}) |$. The base case $| V(\tilde{\Gamma}) |=1$ is vacuous. We now assume the statement is true for $| V(\tilde{\Gamma}) |< m$ with $m\geq2$. Let the source graph $\tilde{\Gamma}$ of $\tilde{h}$ have exactly $m$ vertices. We cut at a bounded edge $F\in E_f(\tilde{\Gamma})$ yielding two rigid parametrised tropical curves $$\tilde{h}_1,\tilde{h}_2:\tilde{\Gamma}_1,\tilde{\Gamma}_2 \rightarrow \mathbb{R}^2.$$ Associated to $\tilde{h}_1,\tilde{h}_2$ are moduli spaces $M_{\tilde{h}_i}$ of stable logarithmic maps marked by $\tilde{h}_i$. Consider the following diagram
$$
\begin{tikzcd}
\mathsf{M}_{g, \Delta}^{\tilde{h}} \arrow[d, "\mathsf{c}"]& & \\
\mathsf{M}_{\tilde{h}_1}\times \mathsf{M}_{\tilde{h}_2} \arrow[d, "\mathsf{ev}_F"] \arrow[r, "\mathsf{cut}^1\times \mathsf{id}"]                            & \prod_{V\in V(\tilde{\Gamma}_1)} \mathsf{M}'_{g_V,\Delta_V} \times \mathsf{M}_{\tilde{h}_2} \arrow[r, "\mathsf{id}\times \mathsf{cut}^2"] \arrow[d, "\mathsf{ev}_F'"]                 & \prod_{V\in V(\tilde{\Gamma}_1)} \mathsf{M}'_{g_V,\Delta_V} \times \prod_{V\in V({\tilde{\Gamma}_2})} \mathsf{M}'_{g_V,\Delta_V} \arrow[d, "\ev_{(e)}"] \\
 D_F \times D_F \ar[equal]{r} & D_F \times D_F &  (D_F \times D_F)\times \prod_{E \in E_f(\tilde{\Gamma}_1)} D^2_E \times \prod_{E \in E_f(\tilde{\Gamma}_2)} D^2_E  \arrow[l].          
\end{tikzcd}
$$
By \cite[p.$45$ penultimate sentence]{ranganthan2022logarithmic} we have
$$\mathsf{c}_{\star}([\mathsf{M}_{g, \Delta}^{\tilde{h}}]^{\vir})=w(F) \,\mathsf{ev}_F^{\star}(\kappa_F)\cap \left ([\mathsf{M}_{\tilde{h}_1}]^{\vir} \times [\mathsf{M}_{\tilde{h}_2}]^{\vir} \right).$$
By the inductive hypothesis
\begin{gather*}
(\mathsf{cut}^{1})_{\star}([\mathsf{M}_{\tilde{h}_1}]^{\vir})=\left(\prod_{E \in E_{f}(\tilde{\Gamma}_1)} w(E) \right)\mathsf{ev}_{\tilde{h}_1}^\star(\kappa_{\tilde{h}_1})\cap \prod_{V \in V(\tilde{\Gamma}_1)} [\mathsf{M}'_{g_V, \Delta_V}]^{\vir},\\
(\mathsf{cut}^{2})_{\star}([\mathsf{M}_{\tilde{h}_2}]^{\vir})=\left(\prod_{E \in E_{f}(\tilde{\Gamma}_2)} w(E) \right)\mathsf{ev}_{\tilde{h}_2}^\star(\kappa_{\tilde{h}_2})\cap \prod_{V \in V(\tilde{\Gamma}_2)} [\mathsf{M}'_{g_V, \Delta_V}]^{\vir}.
\end{gather*}
Combining with a diagram chase, we get
\begin{align*}
\mathsf{cut}_{\star}([\mathsf{M}_{g, \Delta}^{\tilde{h}}]^{\vir})=&\left(\prod_{E \in E_{f}(\tilde{\Gamma})} w(E) \right)\mathsf{ev}_{(e)}^\star(\kappa_{F}\otimes 1 \otimes 1) \mathsf{ev}_{(e)}^\star(1\otimes \kappa_{\tilde{h}_1}\otimes 1) \mathsf{ev}_{(e)}^\star(1 \otimes 1\otimes \kappa_{\tilde{h}_2})\cap \prod_{V \in V(\tilde{\Gamma})} [\mathsf{M}'_{g_V, \Delta_V}]^{\vir}\\=&
\left(\prod_{E \in E_{f}(\tilde{\Gamma})} w(E) \right)\mathsf{ev}_{(e)}^{\star}(\kappa_{\tilde{h}})\cap \prod_{V \in V(\tilde{\Gamma})} [\mathsf{M}'_{g_V,\Delta_V}]^{\vir}.
\end{align*} This completes the proof.
\end{proof}

\begin{proposition}\label{prop:degenerated}
    \begin{equation*}
        N_{g,\tilde{h}}^{\Delta,\textbf{k}} = \prod_{E \in E_{f}(\tilde{\Gamma})} w(E)\left(\prod_{V \in V^{3}(\tilde{\Gamma})} N_{g_V,V} \prod_{V \in V^{2\mathsf{u}}(\tilde{\Gamma})} N_{g_V,V} \prod_{V \in V^{2\mathsf{p}} (\tilde{\Gamma})} N_{g_V,V} \prod_{V \in V^{\mathsf{mp}}(\tilde{\Gamma})} N_{g_V,V}\right)
    \end{equation*}
\end{proposition}
\begin{proof} 
The integrand of $N_{g,\Delta}^{\tilde{h},\textbf{k}}$ can be written $$(-1)^g \lambda_g \prod_{i=1}^n \psi^{k_i}_i = \mathsf{cut}^\star \left( \prod_{V \in V(\tilde{\Gamma})}(-1)^{g_V} \lambda_{g_V} \prod_{i=1}^n \psi_i^{k_i}\right)$$ 
since $\lambda_g$ classes can be glued ~\cite[Lemma 7]{bousseau2018tropical} and $\psi$ classes pull back under gluing. Combining with proposition \ref{prop:EqualityChow} we learn,

\begin{equation*}
        N_{g,\Delta}^{\tilde{h},\textbf{k}} = \left(\prod_{E \in E_{f}(\tilde{\Gamma})} w(E)\right)\left(\int_{\prod_{V}[\mathsf{M}'_{g_V,\Delta_V}]^{\vir}} \ev_{(e)}^\star (\kappa_{\tilde{h}}) \prod_{V \in V(\tilde{\Gamma})}(-1)^{g_V} \lambda_{g_V} \prod_{i=1}^n \psi_i^{k_i}\right)
    \end{equation*}
By definition of the virtual class on $\mathsf{M}'_{g_V,\Delta_V}$ we have, abusing notation, that 
\begin{equation}\label{almostvertexproduct}
    N_{g,\Delta}^{\tilde{h},\textbf{k}} = \left(\prod_{E \in E_{f}(\tilde{\Gamma})} w(E)\right)\left(\int_{\prod_{V}[\mathsf{M}_{g_V,\Delta_V}]^{\vir}} \prod_{V \in V^{\mathsf{mp}}(\tilde{\Gamma})\cup V^{2\mathsf{p}}(\tilde{\Gamma})} \ev_V^\star(\pt)\ev_{(e)}^\star (\kappa_{\tilde{h}}) \prod_{V \in V(\tilde{\Gamma})}(-1)^{g_V} \lambda_{g_V} \prod_{i=1}^n \psi_i^{k_i}\right)
\end{equation}
    \begin{claim}\label{claim:}
        In equation \eqref{almostvertexproduct} we may replace $\ev_{(e)}^\star(\kappa_{\tilde{h}})$ by $$\prod_{V \in V^{3}(\tilde{\Gamma})} \ev_{\mathrm{in},1}^\star(\pt_{D_1}) \cdot \ev_{\mathrm{in},2}^\star(\pt_{D_2}) \prod_{V \in V^{2\mathsf{u}}(\tilde{\Gamma})} \ev_{\mathsf{in}}^\star (\pt_{D_{\mathsf{in}}}).$$
    \end{claim}
    After proving Claim \ref{claim:} the proposition follows from rearranging equation \eqref{almostvertexproduct}. 
    \begin{proof}
    Certainly, $$\ev_{(e)}^\star(\kappa_{\tilde{h}}) = \prod_{E \in E_f(\tilde{\Gamma})} (\ev_{V^s_E}^E)^\star \pt_E + (\ev_{V^t_E}^E)^\star \pt_E,$$ where $V_E^s,V_E^t$ are the source and target vertices with respect to the orientation on $\tilde{\Gamma}$. By the induction argument of ~\cite[Proposition 22]{bousseau2018tropical}, only the summand $(\ev_{V^t_E}^E)^\star \pt_E$ contributes to our integral. A minor adaptation of op cit is required in our situation: we must check it applies to edges whose source is an $m$-valent pointed vertex. Indeed, if $V \in V^{\mathrm{mp}}(\tilde{\Gamma})$, and $E$ an outgoing edge of $V = V^s$, then $(\ev_{V^s_E}^E)^\star \pt_E$ cannot contribute to the integral because the corresponding term would involve an integral of $$\lambda_{g_V}\ev_V^\star\pt \cdot (\ev_{V^s_E}^E)^\star \pt_E \cdot \psi^{k_i}$$ but this has degree greater than the virtual dimension of $\mathsf{M}_{g_V,\Delta_V}$. 
    \end{proof} 
\end{proof}
Proposition \ref{prop:degenerated} expresses $N_{g,\tilde{h}}^{\Delta,\textbf{k}}$ as a product over all of the vertices of $\tilde{\Gamma}$, the graph appearing in $\tilde{h} : \tilde{\Gamma} \rightarrow \RR^2 \in T^{g,\textbf{k}}_{\Delta,p}$. On the other hand, we want to relate this to a product over the vertices of $\Gamma$ - the graph underlying an element of $T^{\textbf{k}}_{\Delta,p}$. In \cite{bousseau2018tropical}, the author proved that the contributions from the `extra vertices' of $\tilde{\Gamma}$ exactly cancel out the contributions of the weights coming from the `extra edges' of $\tilde{\Gamma}$. More precisely,

\begin{proposition}[{\cite[Corollary 16]{bousseau2018tropical}}]\label{prop : endgluing}
    Assume $\tilde{h} : \tilde{\Gamma} \rightarrow \RR^2$ is an element of $T^{g,\textbf{k}}_{\Delta,p}$. 
    \begin{enumerate}[(1)]
        \item If there exists a bivalent vertex $V$ of $\tilde{\Gamma}$ with $g_V \neq 0$ then $$N_{g,\Delta}^{\tilde{h},\textbf{k}} = 0.$$
        \item  If $g_V = 0$ for all the bivalent vertices $V$ of $\Gamma$, then $$N_{g,\Delta}^{\tilde{h},\textbf{k}} = \left(\prod_{E \in E_{f}(\Gamma)}w(E)\right)\left(\prod_{V \in V^{3}(\tilde{\Gamma})}N_{g_V,V}\right)\left(\prod_{V \in V^{\mathsf{mp}}(\tilde{\Gamma})}N_{g_V,V}\right).$$
    \end{enumerate}
\end{proposition}

\section{Logarithmic Gromov--Witten theory and double ramification cycles}\label{Section : LogGW and DR}
Following Proposition \ref{prop : endgluing}, computing $N_{g,\Delta}^{\tilde{h},\textbf{k}}$ amounts to working out the {$m$-valent pointed vertex contibutions} defined in Section \ref{def : higher valency contribution} $$N_{g_V,V}= \int_{[\mathsf{M}_{g_V,\Delta_V}]^\mathsf{vir}} (-1)^{g_V}\lambda_{g_V}\ev_V^\star (\pt)\psi^{m-2}$$ as the trivalent contributions of Section \ref{def : trivalent contribution} are calculated in \cite{bousseau2018tropical}. 

These numbers are logarithmic Gromov--Witten invariants of a toric surface. In this section we show these Gromov--Witten invariants are equal to integrals of tautological classes on $\overline{\mathcal{M}}_{g,n}$ against a product of double ramification cycles. The crucial step for this result is using machinery developed in \cite[Section 3.4]{AjithDhruv} to turn integrals against the virtual class of the moduli space of stable logarithmic maps into integrals against the toric contact cycle, a higher rank generalisation of the double ramification cycle. We use the same technique to reprove Bousseau's formula for the trivalent unpointed vertex contribution.

\subsection{The double ramification cycle}\label{sec:DoubleRamif} Fix an ordered tuple of integers with sum zero $\textbf{a}=(a_1,\ldots,a_n)$. To define the double ramification cycle we study the space of maps to the unique proper toric variety of dimension one: $\mathbb{P}^1$. There is a moduli space $\mathsf{M}^\mathsf{rub}_{g,\mathbf{a}}(\mathbb{P}^1)$ of equivalence classes of relative stable map to $\mathbb{P}^1$ for which the $i^\mathsf{th}$ marked point has contact order $a_i$ \cite{MarcusWise}. Two maps are identified if there is a torus automorphism of $\mathbb{P}^1$ carrying one map to the other. There is a forgetful map $$\pi:\mathsf{M}^\mathsf{rub}_{g,\mathbf{a}}(\mathbb{P}^1) \rightarrow \overline{\mathcal{M}}_{g,n}$$ which forgets the map and stabilises the underlying curve. The moduli space $\mathsf{M}^\mathsf{rub}_{g,\mathbf{a}}(\mathbb{P}^1)$ carries a virtual cycle $[\mathsf{M}_{g, {\mathbf{a}}}^{\mathsf{rub}}]^\mathsf{vir}$ \cite{MarcusWise}. The \textit{double ramification cycle} $\mathsf{DR}({\mathbf{a}})$ is the class $$\pi_\star \left([\mathsf{M}_{g, {\mathbf{a}}}^{\mathsf{rub}}]^\mathsf{vir}\right)\in \mathsf{A}_{2g-3+n}(\overline{\mathcal{M}}_{g,n}).$$

\subsection{Toric contact cycle} The toric contact cycle is the analogue of the double ramification cycle replacing $\mathbb{P}^1$ by a dimension $r$ toric variety. The contact order of a marked point is now recorded by an integral vector in $\mathbb{Z}^r$. The tuple ${\mathbf{a}}$ is replaced by a $r \times n$ matrix with row sum zero. For us this matrix will always be $\Delta_V$ from the previous section. 

\subsubsection{Rubber stable maps} There is a universal compactification $\mathbb{G}_\mathsf{log}^r$ of an $r$-dimensional torus, see \cite{AjithDhruv} for background. This is a stack on the category of logarithmic schemes and admits every choice of two dimensional toric variety as a \textit{subdivision}. See \cite[Section 1.1]{AjithDhruv} for the definition of subdivision.
There is a moduli space $\mathsf{M}_{g,\Delta_V}^\mathsf{rub}(\Glog^r)$ of stable logarithmic maps to rubber $\Glog^r$ with contact data $\Delta_V$ tracking stable maps to $\Glog^r$ up to the action of $\Glog^r$, see \cite{AjithDhruv,RangWise}. 
\begin{remark}
    The functor $\Glog^r$ is not representable by an algebraic stack with logarithmic structure, however the moduli space $\mathsf{M}_{g,\Delta_V}^\mathsf{rub}(\Glog^r)$ is a Deligne--Mumford stack with logarithmic structure.
\end{remark}

\subsubsection{Virtual fundamental class} Consider the cartesian diagram, $$ 
\begin{tikzcd}
\mathsf{M}_{g,\Delta_V}^\mathsf{rub}(\Glog^r) \arrow[d] \arrow[r] & {\overline{\mathcal{M}}_{g,n}} \arrow[d, "0\times  \ldots \times 0"] \\
{\mathsf{M}^\mathsf{rub}_{g,\Delta_V}}(\mathbb{G}_\mathsf{trop}^r) \arrow[r]                            & {\mathsf{Pic}_{g,n}\times\ldots\times \mathsf{Pic}_{g,n}.}    
\end{tikzcd}$$ The lower horizontal map is the Abel Jacobi section and $0$ denotes the section corresponding to the trivial line bundle. The fibre product in the bottom right involves $r$ copies of $\mathsf{Pic}_{g,n}$ and is taken over $\overline{\mathcal{M}}_{g,n}$. The moduli space $\mathsf{M}_{g,\Delta_V}^\mathsf{rub}(\Glog^r)$ admits a virtual fundamental class defined as $$[\mathsf{M}_{g,\Delta_V}^\mathsf{rub}(\Glog^r)]^\mathsf{vir}:=(0\times\ldots\times  0)^![\mathsf{M}_{g,\Delta_V}^\mathsf{rub}(\Gtrop^r)] $$ where upper shriek denotes refined Gysin pullback. See \cite[p.$22-23$]{MarcusWise} for details in the case $r=1$ and \cite[Section 3.3.3]{AjithDhruv} in general.

\subsubsection{Toric contact cycle} We define the \textit{toric contact cycle}, 
$$
\mathsf{TC}_g(\Delta_V):=\pi_{\star}( [ \mathsf{M}_{g, {\Delta_V}}^{\mathsf{rub}}]^\mathsf{vir})\in \mathsf{A}^{\mathsf{r}g}(\overline{\mathcal{M}}_{g,n}),
$$
where $\pi$ denotes the stabilization morphism $\mathsf{M}_{g, {\Delta_V}}^{\mathsf{rub}}\to \overline{\mathcal{M}}_{g,n}$. {In the literature the toric contact cycle is sometimes called the double double ramification cycle, see \cite{MR4490707,molcho2021case,holmes2022logarithmic} for background and development of the theory}. In the sequel we will mean $r=2$ when we say toric contact cycle.
    Setting $r=1$ the toric contact cycle coincides with the double ramification cycle defined in Section~\ref{sec:DoubleRamif}, see \cite{MarcusWise}.

\subsection{Rubber and rigidified geometry} Each column vector $\delta_i$ of ${\Delta_V}$ determines a toric stratum $Y_i$ of $X_{{\Delta_V}}$ for $i=1, \ldots, n$. The \textit{rigid evaluation space} is the toric variety
$$
\mathsf{Ev}_{{\Delta_V}}:=Y_1 \times \ldots \times Y_n.
$$
In the sequel we assume that the dense torus $T=(\mathbb{C}^{\star})^2$ of $X_{\Delta_V}$ acts effectively on $\mathsf{Ev}_{{\Delta_V}}$. Thus we identify the dense torus of  $X_{\Delta_V}$ as a subtorus $T$ of the dense torus $\mathsf{Ev}_{{\Delta_V}}^\circ$ of $\mathsf{Ev}_{{\Delta_V}}$. There is a smooth toric compactification $\mathsf{Ev}^{\mathsf{rub}}_{{\Delta_V}}$ of $\mathsf{Ev}_{{\Delta_V}}^\circ/T$ such that the following proposition holds.

\begin{proposition}\label{prop:AjithDhruv}(\cite[Section 3.4, Proof of Theorem B]{AjithDhruv}) After possibly replacing $\mathsf{Ev}_{{\Delta_V}}, \mathsf{M}_{g, {\Delta_V}}, \mathsf{M}_{g, {\Delta_V}}^{\mathsf{rub}}$ by a subdivision, there is a commutative diagram with the right hand square cartesian,
$$
\begin{tikzcd}
\mathsf{M}_{g, {\Delta_V}} \arrow[rd, "\epsilon"] \arrow[r, "p"] & P \arrow[r, "\mathsf{ev}"] \arrow[d,swap,"\tilde{\epsilon}"] \arrow[rd, phantom, "\square"]
& \mathsf{Ev}_{{\Delta_V}} \arrow[d,"\delta"] \\
& \mathsf{M}_{g, {\Delta_V}}^{\mathsf{rub}}(\Glog^2)\arrow[r,swap,"\mathsf{ev}_{\mathsf{rub}}"]
&  \mathsf{Ev}_{ {\Delta_V}}^{\mathsf{rub}}.
\end{tikzcd}$$
Both $\tilde{\epsilon}$ and $\delta$ are flat and proper. The morphism $\delta$ is toric and on the level of tori restricts to the quotient map $$\mathsf{Ev}_{\Delta_V}^\circ \rightarrow \mathsf{Ev}_{\Delta_V}^\circ/T.$$
\end{proposition}

\begin{corollary}
For $\gamma$ an element of $\mathsf{A}^\star(\mathsf{Ev}_{{\Delta_V}})$ there is an equality
$$
\epsilon_\star(p^\star \mathsf{ev}^\star(\gamma) \cap [\mathsf{M}_{g, {\Delta_V}}]^\mathsf{vir})= \mathsf{ev}_\mathsf{rub}^\star(\delta_\star(\gamma)) \cap [\mathsf{M}^{\mathsf{rub}}_{g, {\Delta_V}}(\Glog^2)]^\mathsf{vir}.
$$
\end{corollary}
\begin{proof}
     From \cite[Theorem 3.3.2]{AjithDhruv} we get $[\mathsf{M}_{g, {\Delta_V}}]^\mathsf{vir} = {\epsilon}^\star[\mathsf{M}^{\mathsf{rub}}_{g, {\Delta_V}}(\Glog^2)]^\mathsf{vir}$. Now appeal to the fact that the cartesian square in Proposition \ref{prop:AjithDhruv} is Cartesian. 
\end{proof}
\subsection{Vertex contributions as integrals on $\overline{\mathcal{M}}_{g,n}$} Write $\pt_{D_i}$ for the class of a point on $D_i$ in $\mathsf{A}^\star(\mathsf{Ev}_{\Delta_V})$.
\subsubsection{$m$-valent pointed vertices} We require a preparatory Lemma. Let $n$ be a natural number and choose a morphism $$\varphi:\mathbb{G}_{\mathsf{log}}^2 \rightarrow \mathbb{G}_{\mathsf{log}}^{n-2} \textrm{ and write } \varphi'=\mathsf{id}\times \varphi: \mathbb{G}_{\mathsf{log}}^2 \rightarrow \mathbb{G}_{\mathsf{log}}^n.$$

\begin{lemma}\label{lem:SquareCartesian}
    There is a Cartesian square 
       \begin{equation}\label{eqn:basicCartesianDiagram}
\begin{tikzcd}
\Glog^{n} \arrow[r,"p"] \arrow[d]\arrow[rd, phantom, "\square"]                 & \mathbb{G}_\mathsf{log}^2 \arrow[d] \\
\Glog^n/ \varphi'(\mathbb{G}_\mathsf{log}^2) \arrow[r] & \mathsf{Spec}(\mathbb{C}).          \end{tikzcd}\end{equation}
\end{lemma}

\begin{proof}
        There is certainly a map $$\alpha\colon \Glog^n\rightarrow \Glog^n/ \varphi'(\mathbb{G}_\mathsf{log}^2)\times \mathbb{G}_\mathsf{log}^2$$ and our task is to verify this map is an isomorphism. For $T$ an arbritary test scheme equipped with sheaf of monoids $M_T$ we have $$\alpha_T:\mathsf{Hom}(T,\Glog^n) = \Gamma(T,M_T)^n \rightarrow \mathsf{Hom}(T,\Glog^n/ \varphi'(\mathbb{G}_\mathsf{log}^2)\times \mathbb{G}_\mathsf{log}^2).$$ A point in the set on the right hand side is specified by an element of $(h_1,h_2) \in \Gamma(T,M_T)^n\times \Gamma(T,M_T)^2$, but tuples $(h_1,h_2), (h_1',h_2')$ correspond to the same point of $\mathsf{Hom}(T,\Glog^n/ \varphi'(\mathbb{G}_\mathsf{log}^2)\times \mathbb{G}_\mathsf{log}^2)$ if and only if $h_2= h_2'$ and $h_1^{-1}h_1'$ lies in the image of $\varphi'$. Elementary group theory shows $\alpha_T$ is a bijection and we deduce $\alpha$ is an isomorphism.
\end{proof}

\begin{lemma}\label{lem:mValentMovetoTC}
    Let $V$ be a pointed $m$-valent vertex. We then have
    $$
    \pi_{\star}(\ev_V^\star (\pt) \cap [\mathsf{M}_{g_V,\Delta_V}(\Glog^2)]^\mathsf{vir}) = \mathsf{TC}_g({\Delta_V})
    $$
\end{lemma}
\begin{proof}
Since $\mathsf{Ev}_{\Delta_V}$ is toric we may regard it as a subdivision of $\Glog^k$ for some $k$. The map $\delta$ can then be understood as obtained from a quotient map $$\delta^\circ:\Glog^k\rightarrow\Glog^k/\Glog^2$$ by compatible subdivision of source and target. 

We now think of $\delta^\circ$ as the left vertical arrow of the diagram in equation \eqref{eqn:basicCartesianDiagram}. Choose a subdivision of every space in equation \eqref{eqn:basicCartesianDiagram} such that, after replacing $\mathsf{Ev}_{{\Delta_V}}$ by a subdivision, we have a cartesian square \begin{equation}\label{eqn:SubdivGlog}
\begin{tikzcd}
\mathsf{Ev}_{{\Delta_V}} \arrow[r,"p"] \arrow[d] \arrow[rd, phantom, "\square"]                & W \arrow[d,"\delta"] \\
\mathsf{Ev}_{{\Delta_V}}^\mathsf{rub} \arrow[r] & \mathsf{Spec}(\mathbb{C})          \end{tikzcd}.\end{equation} Here $W$ is a toric variety of dimension two. Replacing $P$ and $\mathsf{M}_{g, {\Delta_V}}^{\mathsf{rub}}(\Glog^2)$ by a subdivision, we may concatenate the Cartesian squares in equations \ref{eqn:basicCartesianDiagram} and \ref{eqn:SubdivGlog},
$$
\begin{tikzcd}
P \arrow[d] \arrow[r]\arrow[rd, phantom, "\square"]                             & \mathsf{Ev}_{{\Delta_V}}\arrow[rd, phantom, "\square"] \arrow[r] \arrow[d]                 & W \arrow[d] \\
{\mathsf{M}_{g, {\Delta_V}}^{\mathsf{rub}}}(\Glog^2) \arrow[r] & \mathsf{Ev}_{{\Delta_V}}^\mathsf{rub} \arrow[r] & \mathsf{Spec}(\mathbb{C}).          
\end{tikzcd}.$$
By abstract nonsense we have built a cartesian square of logarithmic schemes which fits into the larger diagram $$ 
\begin{tikzcd}
{\mathsf{M}_{g,\Delta_V}} \arrow[rrd, bend left, "\mathsf{ev}_1"] \arrow[rd,"t"] \arrow[rdd, bend right, "\epsilon"] &                                                     &             \\ & P \arrow[r,"q"] \arrow[d, "\overline{\epsilon}"] \arrow[rd, phantom, "\square"]                             & W \arrow[d,"r"]               \\         & {{\mathsf{M}_{g, \Delta_V}^{\mathsf{rub}}}(\Glog^2)} \arrow[r,"s"] & \mathsf{Spec}(\mathbb{C})
\end{tikzcd}$$ Let $\mathsf{pt}$ be the cohomology class dual to a point in the dense torus of $W$. There is an equality of Chow cycles $$\overline{\epsilon}_\star q^\star \mathsf{pt} = s^\star r_\star\mathsf{pt} = 1.$$
Applying both sides to $[\mathsf{M}_{g, {\Delta_V}}^\mathsf{rub}(\Glog^2)]^\mathsf{vir}$ we obtain $$\overline{\epsilon}_\star \left(q^\star \mathsf{pt}\cap \overline{\epsilon}^\star[\mathsf{M}_{g, {\Delta_V}}^\mathsf{rub}(\Glog^2)]^\mathsf{vir}\right)=\overline{\epsilon}_\star q^\star \mathsf{pt}\cap [\mathsf{M}_{g, {\Delta_V}}^\mathsf{rub}(\Glog^2)]^\mathsf{vir} = [\mathsf{M}_{g, {\Delta_V}}^\mathsf{rub}(\Glog^2)]^\mathsf{vir}, $$ where the first equality is from the projection formula. To complete the proof, push this formula forward to $\overline{\mathcal{M}}_{g,n}$ along the map $$\pi':\mathsf{M}_{g, {\Delta_V}}^\mathsf{rub}(\Glog^2) \rightarrow \overline{\mathcal{M}}_{g,n}$$ and observe $\pi = \pi'\circ \epsilon$ and by use of \cite[Theorem 3.3.2]{AjithDhruv} we also have $\overline{\epsilon}^\star[\mathsf{M}_{g, {\Delta_V}}^\mathsf{rub}(\Glog^2)]^\mathsf{vir} = t_{\star}[\mathsf{M}_{g, {\Delta_V}}]^\mathsf{vir}$.
\end{proof}

\subsubsection{Trivalent unpointed vertices}
\begin{lemma}\label{lem:3valentmovetoTC}
    Let $V$ be a trivalent unpointed vertex. Let $\vec{v}_1$ and $\vec{v}_2$ denote the primitive generators of the rays corresponding to $D_1$ and $D_2$.
    $$
    \pi_{\star}\left(\ev^\star (\pt_{D_1})\ev^\star (\pt_{D_2}) \cap [\mathsf{M}_{g_V,\Delta_V}(\Glog^2)]^\mathsf{vir}\right)= |\vec{v}_1\wedge \vec{v}_2| \mathsf{TC}_g({\Delta_V}),
    $$
    where $\pi$ denotes the forget and stabilise morphism.
\end{lemma}
\begin{proof}
    The proof involves explicit computation so we fix coordinates. Consider the morphism $$\delta: \mathsf{Ev}\rightarrow \mathsf{Ev}^\mathsf{rub}.$$ Restricting to dense tori $\delta$ is the quotient map $$\delta^\circ: (\mathbb{C}^\star)^3\rightarrow (\mathbb{C}^\star)^3/ (\CC^\star)^2.$$ Here the action of $(\mathbb{C}^\star)^2$ is the action of the dense torus of $X$ on its toric boundary strata. Therefore, on the level of cocharacters the map $\delta^\circ$ is specified by quotienting by the column span $W$ of the matrix with rows $v_i$. Passing to a subdivision if necessary, we may assume $\mathsf{Ev}$ is smooth. The cohomology class $\mathsf{pt}_{D_1}\cup \mathsf{pt}_{D_2}$ is poincare dual to the closure $Z$ in $\mathsf{Ev}$ of $V(X-1,Y-1)$. Note $Z$ is just a copy of $\mathbb{P}^1$.

    Restricting $\delta^\circ$ or $\delta$ to the dense torus in $Z$ yields the map $\CC^\star \rightarrow \CC^\star$ which factors as $$\mathbb{C}^\star \rightarrow (\mathbb{C}^\star)^3\rightarrow \mathbb{C}^\star.$$ On the level of cocharacters these maps are $$1 \mapsto (0,0,1)\textrm{ and } (a,b,c) \mapsto [(a,b,c)]\in Z^3/W.$$
    This extends to a map $\mathbb{P}^1 \rightarrow \mathbb{P}^1$ of degree $|\vec{v}_1\wedge \vec{v}_2|$ because this is the factor by which cocharacters are scaled. It follows from definitions that on the level of cycles $$\delta_\star[Z] = |\vec{v}_1\wedge \vec{v}_2|[\mathbb{P}^1].$$ The rest of the proof follows the steps in the proof of Proposition \ref{lem:mValentMovetoTC}.
\end{proof}

\subsection{Double ramification and toric contact cycles} 
We have expressed vertex contributions as integrals over the toric contact cycle, and it remains to compute these integrals. In this section we establish Proposition \ref{DDRvsDRDR} which asserts that in the presence of a $\lambda_g$ class, the toric contact cycle is a product of double ramification cycles.

\subsubsection{The compact type locus} Define an open subscheme $$j:\mathcal{M}_{g,n}^\mathsf{ct}\hookrightarrow \overline{\mathcal{M}}_{g,n}$$ called the \textit{compact type} locus parametrising curves whose arithmetic and geometric genus coincide. The preimage of $\mathcal{M}_{g,n}^\mathsf{ct}$ in $\mathsf{M}^\mathsf{rub}_{g,\Delta_V}(\Glog^2)$ is denoted $k:\mathsf{M}^\mathsf{rub,ct}_{g,\Delta_V}(\Glog^2)\hookrightarrow \mathsf{M}^\mathsf{rub}_{g,\Delta_V}(\Glog^2).$ Similarly define $$k_x:\mathsf{M}^\mathsf{rub,ct}_{g,\Delta_V^x}(\Glog)\hookrightarrow \mathsf{M}^\mathsf{rub}_{g,\Delta_V^x}(\Glog)\textrm{ and } k_y: \mathsf{M}^\mathsf{rub,ct}_{g,\Delta_V^y}(\Glog)\hookrightarrow \mathsf{M}^\mathsf{rub}_{g,\Delta_V^y}(\Glog).$$ Recall ${\Delta_V}$ is a ${2 \times n}$ matrix of balanced contact order data with rows ${\Delta_V^x},{\Delta_V^y}$. 
\begin{proposition}\label{prop:Restriction vanishes}
    There is an equality in the Chow group of $\mathcal{M}_{g,n}^\mathsf{ct}$, $$ j^\star \left(\mathsf{TC}_g({\Delta_V}) - \mathsf{DR}_g({\Delta_V^x})\mathsf{DR}_g({\Delta_V^y})\right)=0.$$
\end{proposition}

The proof of Proposition \ref{prop:Restriction vanishes} requires preparatory lemmas. We have maps $$\pi: \mathsf{M}^\mathsf{rub}_{g,\Delta_V}(\Glog^2) \rightarrow \overline{\mathcal{M}}_{g,n}, \quad\pi^x: M_{g,\Delta^x_V}^\mathsf{rub}(\Glog) \rightarrow \overline{\mathcal{M}}_{g,n},\quad \pi^y: \mathsf{M}^\mathsf{rub}_{g,\Delta^y_V}(\Glog) \rightarrow \overline{\mathcal{M}}_{g,n}.$$ These maps restrict to define $$\overline{\pi}: \mathsf{M}^\mathsf{rub,ct}_{g,\Delta_V}(\Glog^2) \rightarrow \mathcal{M}_{g,n}^{\mathsf{ct}}, \quad\overline{\pi}_x: \mathsf{M}^\mathsf{rub,ct}_{g,\Delta^x_V}(\Glog) \rightarrow \mathcal{M}_{g,n}^{\mathsf{ct}},\quad \overline{\pi}_y: \mathsf{M}^\mathsf{rub,ct}_{g,\Delta^y_V}(\Glog) \rightarrow \mathcal{M}_{g,n}^{\mathsf{ct}}.$$ 
Consider the diagram \begin{equation}\label{eqn:DoubleSquare}
\begin{tikzcd}
\mathsf{M}_{g,\Delta}^\mathsf{rub,ct}(\Glog^2) \arrow[d, "\pi"] \arrow[r]\arrow[rd, phantom, "\square"] & \mathsf{M}_{g,\Delta^x}^\mathsf{rub,ct}(\Glog)\times \mathsf{M}^\mathsf{rub,ct}_{g,\Delta^y}(\Glog)\arrow[rd, phantom, "\square"] \arrow[d, "\pi_x\times \pi_y"] \arrow[r] & {\mathcal{M}_{g,n}^{\mathsf{ct}}\times \mathcal{M}_{g,n}^{\mathsf{ct}} } \arrow[d, "0 \times 0"] \\
{\mathcal{M}_{g,n}^{\mathsf{ct}}} \arrow[r, "\kappa"]             & {\mathcal{M}_{g,n}^{\mathsf{ct}}\times \mathcal{M}_{g,n}^{\mathsf{ct}} } \arrow[r]                         & {\mathsf{Pic}_{g,n}^{\mathsf{ct}}\times \mathsf{Pic}_{g,n}^{\mathsf{ct}}}   
\end{tikzcd}.
\end{equation}
The first horizontal map is the diagonal embedding; the second is the Abel--Jacobi section on each factor. Define a class $$ [\mathsf{M}^{\mathsf{rub,ct}}_{g,\Delta}(\Glog^2)]^{\mathsf{vir}}= (0 \times 0)^![\mathcal{M}_{g,n}^{\mathsf{ct}}].$$

\begin{lemma}\label{lem:Prop on CT locus}
    $$\kappa^\star(\pi^x_\star[\mathsf{M}_{g,\Delta^x}^\mathsf{rub,ct}(\Glog)]^\mathsf{vir}\times \pi^y_\star [\mathsf{M}_{g,\Delta^y}^\mathsf{rub,ct}(\Glog)]^\mathsf{vir})=\pi_\star[\mathsf{M}_{g,\Delta}^\mathsf{rub,ct}(\Glog^2)]^\mathsf{vir}$$
\end{lemma}
\begin{proof}
    Applying \cite[Theorem 6.4]{Fulton} to diagram \ref{eqn:DoubleSquare} we learn $$\kappa^!([\mathsf{M}_{g,\Delta^x}^\mathsf{rub,ct}(\Glog)]^\mathsf{vir}\times [\mathsf{M}_{g,\Delta^y}^\mathsf{rub,ct}(\Glog)]^\mathsf{vir})=[\mathsf{M}_{g,\Delta}^\mathsf{rub,ct}(\Glog^2)]^\mathsf{vir}.$$ Push this equality forward along the map $\pi$ to complete the proof.
\end{proof}

\begin{lemma}\label{lem: Pullback plays well with CT}
    There is an equality $$k^\star[\mathsf{M}^\mathsf{rub}_{g,\Delta_V}(\Glog^2)]^\mathsf{vir} =[\mathsf{M}^\mathsf{rub,ct}_{g,\Delta_V}(\Glog^2)]^\mathsf{vir},$$ and similarly for $k_x,k_y$.
\end{lemma}
\begin{proof}
    We prove the statement for $k$ and note $k_x,k_y$ follow similarly. Consider the commutative diagram in which all squares are cartesian $$\begin{tikzcd}[row sep={40,between origins}, column sep={60,between origins}]
      & \mathsf{M}^{\mathsf{rub,ct}}_{g,\Delta}(\Glog^2) \arrow[rr,"k"]\ar{dd}\arrow[dl] & & \mathsf{M}^{\mathsf{rub}}_{g,\Delta}(\Glog^2)\vphantom{\times_{S_1}} \ar{dd}\ar{dl} \\
    \mathcal{M}_{g,n}^\mathsf{ct}\times \mathcal{M}_{g,n}^\mathsf{ct} \ar[crossing over]{rr} \ar{dd} & & \mathcal{M}_{g,n}\times \mathcal{M}_{g,n} \\
     & \mathcal{M}_{g,n}^\mathsf{ct} \ar{rr} \ar{dl} & &\mathsf{M}^{\mathsf{rub}}_{g,\Delta}(\mathbb{G}_\mathsf{trop}^2)   \vphantom{\times_{S_1}} \ar{dl} \\
    \mathsf{Pic}_{g,n}^\mathsf{ct}\times \mathsf{Pic}^\mathsf{ct}_{g_n} \ar{rr} && \mathsf{Pic}_{g,n}\times \mathsf{Pic}_{g,n}. \ar[from=uu,crossing over]
\end{tikzcd} $$ Observe the map $\mathcal{M}^{\mathsf{ct}}_{g,n} \rightarrow \mathsf{M}^{\mathsf{rub}}_{g,\Delta}(\mathbb{G}_\mathsf{trop}^2)$ is an open immersion, and thus the result follows from \cite[Theorem 6.2 (b)]{Fulton}. 
\end{proof}
\begin{lemma}\label{lem:Yoga of CT}
    There is an equality in the Chow group of $\overline{\mathcal{M}}_{g,n}$, $$j^\star\pi_\star[\mathsf{M}^\mathsf{rub}_{g,\Delta_V}(\Glog^2)]^\mathsf{vir}  = \overline{\pi}_\star[\mathsf{M}^\mathsf{rub,ct}_{g,\Delta_V}(\Glog^2)]^\mathsf{vir}.$$
\end{lemma}
\begin{proof}
    Note $\mathsf{M}^{\mathsf{rub}}_{g,\Delta}(\mathbb{G}_\mathsf{trop}^2)$ is an open subscheme in a subdivision $\tilde{\mathcal{M}}_{g,n}$ of $\overline{\mathcal{M}}_{g,n}$. We thus have a commutative diagram, $$
\begin{tikzcd}
\mathsf{M}^\mathsf{rub}_{g,\Delta}(\Glog^2) \arrow[d] \arrow[dd, "\pi"', bend right]\arrow[rd, phantom, "\square"] & {\mathsf{M}^\mathsf{rub,ct}_{g,\Delta}(\Glog^2)} \arrow[d, "\overline{\pi}"] \arrow[l, "k"'] \\
{\tilde{\mathcal{M}}_{g,n}} \arrow[d]                                             & {\mathcal{M}_{g,n}^\mathsf{ct}} \arrow[ld, "j"] \arrow[l]                               \\
{\overline{\mathcal{M}}_{g,n}}.                                                    &                                                                                
\end{tikzcd}$$ The result now follows combining Lemma \ref{lem: Pullback plays well with CT} and \cite[TAG 0EPD]{stacks-project}.
\end{proof}

\begin{proof}[Proof of Proposition \ref{prop:Restriction vanishes}]
By Lemma \ref{lem:Prop on CT locus} we know $$\pi_\star[\mathsf{M}_{g,\Delta}^\mathsf{rub,ct}(\Glog^2)]^\mathsf{vir} - \pi^x_\star[\mathsf{M}_{g,\Delta^x}^\mathsf{rub,ct}(\Glog)]^\mathsf{vir} \cdot \pi^y_\star[\mathsf{M}_{g,\Delta^y}^\mathsf{rub,ct}(\Glog)^\mathsf{vir}]=0$$ in the Chow group of $\cM^\mathsf{ct}_{g,n}$. Combining with Lemma \ref{lem:Yoga of CT} this equation becomes $$ j^\star \left(\pi_\star[\mathsf{M}_{g,\Delta}^\mathsf{rub}(\Glog^2)]^\mathsf{vir} - \pi^x_\star[\mathsf{M}_{g,\Delta^x}^\mathsf{rub}(\Glog)]^\mathsf{vir} \cdot \pi^y_\star[\mathsf{M}_{g,\Delta^y}^\mathsf{rub}(\Glog)^\mathsf{vir}]\right) =0.$$ Substituting in the definition of toric contact cycle and double ramification cycle, the proof is complete. 
\end{proof}
\subsubsection{Connecting toric contact and double ramification cycles} The key result of this subsection follows.
\begin{proposition}\label{DDRvsDRDR}
There is an equality in the Chow group of $\overline{\mathcal{M}}_{g,n}$ $$\lambda_g \mathsf{TC}_g({\Delta_V}) = \lambda_g \mathsf{DR}_g({\Delta_V^x})\mathsf{DR}_g({\Delta_V^y}).$$
\end{proposition}
\begin{proof}
Our task is to establish $$R=\lambda_g \left(\mathsf{TC}_g({\Delta_V}) - \mathsf{DR}_g({\Delta_V^x})\mathsf{DR}_g({\Delta_V^y})\right)=0.$$ 
Combining the excision sequence $$\mathsf{A}_\star (Z) \rightarrow \mathsf{A}_\star(\overline{\mathcal{M}}_{g,n}) \rightarrow \mathsf{A}_\star(\mathcal{M}^\mathsf{ct}_{g,n})$$ with Proposition~\ref{prop:Restriction vanishes} we learn there is some class $R' \in \mathsf{A}_\star(Z)$ which pushes forward to $\mathsf{TC}_g({\Delta_V}) - \mathsf{DR}_g({\Delta_V^x})\mathsf{DR}_g({\Delta_V^y})$. Observe $R = \lambda_g\cap \iota_\star R' = \iota_\star(j^\star\lambda_g \cap R')$. Since $j^\star\lambda_g=0$ \cite[Lemma 7,8]{bousseau2018tropical} we deduce $R=0$ and the proof is complete.
\end{proof}
\subsection{Vertex contributions and double ramification cycles} To finish this section we apply Proposition \ref{DDRvsDRDR} to provide a new formula for the vertex contributions introduced in Section \ref{VertexContributionSub}. 
\begin{lemma}\label{lem:ThreeValentThroughDR}
Let $V$ be a trivalent unpointed vertex and let $\vec{v}_1,\vec{v}_2$ denote the first two columns of $\Delta_V$.
$$
N_{g_V,V}= \mid \vec{v}_1\wedge \vec{v}_2 \mid \int_{\overline{\mathcal{M}}_{g_V, 3}} (-1)^{g_V} \lambda_{g_V} \mathsf{DR}_{g_V}(\Delta_V^x) \mathsf{DR}_{g_V}( \Delta_V^y)
$$
\end{lemma}

\begin{proof}
Starting with the definition of the left hand side,
$$N_{g_V,V} = \int_{[M_{g_V,V}]^\mathsf{vir}} (-1)^{g_V}\lambda_{g_V}\ev_{\mathsf{in},1}^\star (\pt_{D_1})\ev_{\mathsf{in},2}^\star (\pt_{D_2})$$
we apply Lemma \ref{lem:3valentmovetoTC} and the projection formula to obtain, $$N_{g_V,V} = |\vec{v}_1\wedge \vec{v}_2|\int_{\overline{\mathcal{M}}_{g_V, 3}} (-1)^{g_V} \lambda_{g_V}\mathsf{TC}_{g_V}(\Delta_V).$$ The result now follows by Proposition \ref{DDRvsDRDR}.
\end{proof}
\begin{lemma}\label{lem:VertexIsDDR}
 We have an equality
$$N_{g_V,V}= \int_{\overline{\mathcal{M}}_{g_V, m+1}} (-1)^{g_V} \lambda_{g_V} \psi_1^{m-2} \mathsf{DR}_{g_V}(\Delta_V^x) \mathsf{DR}_{g_V}( \Delta_V^y).$$
\end{lemma}
\begin{proof}
Similar to proof of Lemma \ref{lem:ThreeValentThroughDR}, except we apply Lemma \ref{lem:mValentMovetoTC} in place of Lemma~\ref{lem:3valentmovetoTC}.
\end{proof}

\section{Results from the theory of double ramification hierarchies}\label{Section : Integrable Hierachies}
Set $d$ a positive integer and let $\textbf{a}, \textbf{b}$ be tuples of $d+2$ integers with sum zero. Write $\textbf{a}_0,\textbf{b}_0$ the vectors formed from $\textbf{a},\textbf{b}$ by prefixing zero. In this section we evaluate $$I_{g,d}(\textbf{a};\textbf{b}) = \int_{\overline{\mathcal{M}}_{g, d+3}} (-1)^g \lambda_g \psi_1^{d} \mathsf{DR}_g(\textbf{a}_0) \mathsf{DR}_g(\textbf{b}_0).$$ Combining with Lemma \ref{lem:VertexIsDDR} for $d = m-2$ we have computed the $m$-valent pointed vertex contribution $N_{g_V,V}$. 

The geometric input in this computation are the WDVV relations on the Losev--Manin space \cite{LosevManin}. Buryak and Rossi leveraged these equations to study generating series involving $I_{g,d}(\textbf{a};\textbf{b})$ \cite{Buryak,buryak2021quadratic}. We use their analysis to understand $I_{g,d}(\textbf{a};\textbf{b})$. 

We set up notation. Throughout this section  $u_{p,q},p^{a}_b, e^{iy}, e^{ix},\epsilon$ denote formal variables and we often write $u_{0,0} = u$. The heuristic is to study a function on $S^1 \times S^1$ expressed as a fourier series 
$$u=u_{0,0} = \sum_{a,b \in \mathbb{Z}}p^a_b e^{i({ay+bx})}\textrm{ and its derivatives }\partial_x^i \partial_y^j u = u_{i,j}.$$ More formally, define a map \begin{equation*}
    T: \mathbb{C}[[u_{\star,\star}, \epsilon]]\to  \mathbb{C}[[p_\star^{\star}, e^{\pm i x}, e^{\pm i y}, \epsilon]] \textrm{ by sending } u_{k_1, k_2} \mapsto \partial^{k_1}_x\partial^{k_2}_y \left(\sum_{a, b} p_b^{a} e^{i(ay+bx)}\right).
\end{equation*}
For $g$ an element of $\mathbb{C}[[p_\star^\star,\epsilon,e^{\pm i x},e^{\pm i y}]]$ we write $\overline{g}$ for its $e^{i0}$ coefficient considered an element of $\mathbb{C}[[p_\star^\star,\epsilon]]$. This is the algebraic incarnation of an integration map. Moreover set $T_0: \mathbb{C}[[u_{\star,\star}, \epsilon]]\to  \mathbb{C}[[p_\star^{\star}, \epsilon]]$ by composing $T$ with projection to the coefficient of $e^{i0}$. The next definitions, motivated by the chain rule, complete our setup, $$\partial_x,\partial_y:\mathbb{C}[[u_{\star,\star},\epsilon]]\rightarrow \mathbb{C}[[u_{\star,\star},\epsilon]]$$ $$\partial_x:f\mapsto \sum_{k_1, k_2\geq0} u_{k_1+1, k_2} \frac{\partial f}{\partial u_{k_1, k_2}},\quad\quad \partial_y: f\mapsto \sum_{k_1, k_2\geq0} u_{k_1+1, k_2} \frac{\partial f}{\partial u_{k_1, k_2}}.$$

\subsection{The quadratic double ramification integrable hierarchy} Buryak and Rossi study integrals $I_{g,d}(\textbf{a};\textbf{b})$ in the context of integrable hierarchies. We introduce some language from this area so that we may extract Proposition \ref{prop : Igd to cosine}.
\subsubsection{Variational derivative} Write $W$ for the subspace of $\mathbb{C}[[p^\star_\star,\epsilon]]$ defined by image of $T_0$. The variational derivative of $\overline{g}$ in $W\subseteq \mathbb{C}[[p^\star_\star,\epsilon]]$
is defined by choosing $f$ in $\mathbb{C}[[u_{\star,\star},\epsilon]]$ such that $T_0(f) = \overline{g}$. Thinking of $\overline{g}$ as depending on $u$ and its derivatives in this way, we may ask how $\overline{g}$ is affected by changes in $u_{i,j}$. This information is tracked with a \textit{variational derivative} which we define through the formula \begin{equation*}
\frac{\delta \bar{g}}{\delta u}:= \sum_{k_1, k_2\geq0} (-1)^{k_1+k_2}\partial_x^{k_1}\partial_y^{k_2}\frac{\partial f}{\partial u_{k_1, k_2} }.
\end{equation*}
This definition is independent of the choice of $f$.
\begin{lemma}
\label{lem:varder}
For $\bar{g} \in W$ there is an equality
\begin{equation*}
    T\left(\frac{\delta \bar{g}}{\delta u}\right)=\sum_{a,b\in \mathbb{Z}} \frac{\partial \bar{g}}{\partial p_b^a} e^{-i(ay+bx)}.
\end{equation*}  
\end{lemma}
\begin{proof}
%
Recall that a bar over a symbol means take coefficient of $e^{i0}$. We write $\overline{g} = T_0(f)$. Applying the chain rule we learn \begin{align*}\sum_{a,b\in \mathbb{Z}} \frac{\partial (T_0(f))}{\partial p_b^a} e^{-i(ay+bx)} 
&= \sum_{k_1,k_2\geq 0 } \sum_{a,b\in \mathbb{Z}}\overline{\pdv{T(u_{k_1,k_2})}{p^a_b} T\left( \pdv{f}{u_{k_1,k_2}}\right)}e^{-i(ay+bx)} \\ 
&=\sum_{k_1,k_2\geq 0 } \sum_{a,b\in \mathbb{Z}}\overline{\partial_x^{k_1}\partial_y^{k_2}(e^{i(ay+bx)}) T\left( \pdv{f}{u_{k_1,k_2}}\right)}e^{-i(ay+bx)} \\
&= \sum_{k_1,k_2\geq 0 } \sum_{a,b\in \mathbb{Z}}(-i)^{k_1+k_2}a^{k_1}b^{k_2} \overline{T\left( \pdv{f}{u_{k_1,k_2}}\right)e^{-i(ay+bx)}}e^{i(ay+bx)}.\end{align*} 
To complete the proof we establish 
\begin{equation}\label{eq:finalthing} T\left(\partial_x^{k_1}\partial_y^{k_2} h \right)= \sum_{a,b\in \mathbb{Z}}(i)^{k_1+k_2}a^{k_1}b^{k_2} \overline{T\left( h\right)e^{-i(ay+bx)}}e^{i(ay+bx)}\end{equation}
for any $h\in \mathbb{C}[[u_{\star,\star},\epsilon]]$, and apply it to the case $h = \pdv{f}{u_{k_1,k_2}}$. We will write $T(h) = \sum_{a,b \in \mathbb{Z}}h_{a,b} e^{i(ay+bx)}$ so that $\partial_x^{k_1}\partial_y^{k_2} h = \sum_{a,b \in \mathbb{Z}}(i)^{k_1+k_2}h_{a,b} a^{k_1}b^{k_2}e^{i(ay+bx)}$. Equation \eqref{eq:finalthing} can now be deduced by comparing coefficients of $e^{i(ax+by)}$.
\end{proof}

\subsubsection{The noncommutative Moyal product}
Let $f$ and $g$ be in $\mathbb{C}[[u_{\star,\star}, \epsilon]]$ and define the \textit{non-commutative Moyal product},
\begin{equation*}
    f \star g:= \sum_{n\geq 0} \sum_{k_1+k_2=n} \frac{(-1)^{k_2} (i\epsilon)^n}{2^n k_1! k_2!} \partial^{k_1}_x \partial^{k_2}_y (f) \partial^{k_2}_x \partial^{k_1}_y (g) \in \mathbb{C}[[u_{\star,\star}, \epsilon]].
\end{equation*}
\subsubsection{Generating series}
The integrals $I_{g,d}$ can be packaged in a generating function \begin{equation*}
\overline{g_{d}}=\sum_{g\geq0} \frac{(-\epsilon^2)^g}{(d+2)!} \sum_{a_1, \ldots, a_n, b_1, \ldots, b_n} \int_{\overline{\mathcal{M}}_{g, d+3}} \lambda_g \psi_1^d \DR_g(0, a_1, \ldots, a_{d+2}) \DR_g(0, b_1, \ldots, b_{d+2})\prod_{j=1}^{d+2} p_{b_j}^{a_j},
\end{equation*}
considered as an element of $\mathbb{C}[[p^\star_\star, \epsilon]]$. Buryak and Rossi showed that $\overline{g}_d$ lies in the image of $T_0$. In particular it will make sense to take a variational derivative.
\subsubsection{A result of Buryak and Rossi} The variational derivatives of $\overline{g}_d$ are related to the non-commutative Moyal product through the following theorem.

\begin{theorem}[ {\cite[Theorem 4.1] {buryak2021quadratic}}]\label{buryak21quadratic}
    There is an equality $$\partial_x \frac{\delta \bar{g}_d}{\delta u}= \partial_x \left ( \frac{1}{(d+1)!} (u \star \ldots. \star u) \right)$$ where there are $d+1$ copies of $u$ on the right hand side.
\end{theorem}
\begin{remark}
    Integrable hierarchies are systems of differential equations. Originally such systems arose from studying commuting Hamiltonian flows on a symplectic manifold. The equations in an integrable hierarchy take the form $$\textrm{Differential operator } = \textrm{ a certain variational derivative}.$$
    For us, as for Buryak and Rossi, we understand the left hand side as a formal symbol. Buryak has shown that partial cohomological field theories on finite dimensional vector spaces naturally give rise to such an integrable hierarchy \cite{Buryak}. 
    
    Let $V$ be the free vector space with basis $\{e_a\}_{a \in \mathbb{Z}}$ and define maps $V \mapsto H^\star(\overline{\mathcal{M}}_{g,n})$ $$c_{g,n}(e_{b_1},\ldots,e_{b_n}) = \mathsf{DR}_g(b_1,\ldots,b_n).$$ Buryak and Rossi execute a version of Buryak's construction for a cohomological field theory on $V$ built from this data. The resulting \textit{quadratic double ramification hierarchy} consists of equations $$
    \frac{\partial u}{\partial t_d}= \partial_x \frac{\delta \bar{g}_d}{\delta u}.$$ The left hand side is a formal symbol.

    Buryak and Rossi therefore understand Theorem \ref{buryak21quadratic} as follows. The equations of the quadratic double ramification hierarchy coincide with the equations of the hierarchy $$\frac{\partial u}{\partial t_d} = \partial_x \left ( \frac{1}{(d+1)!} (u \star \cdots. \star u) \right).$$ This second hierarchy is called the \textit{dispersionless noncommutative KdV hierachy}. 
\end{remark}

\subsection{Expression for the double ramification integrals}\label{sec:ConnectHierarchy} In the remainder of this section we extract an expression for $I_{g,d}(\textbf{a};\textbf{b})$ recorded in the following proposition.

\begin{proposition}\label{prop : Igd to cosine}
For $a_1, \cdots, a_{d+1}, b_1, \cdots, b_{d+1}\in \mathbb{Z}$ we have
\begin{gather*}
   \sum_{g\geq0}  I_{g,d}( \overline{\textbf{a}}, \overline{\textbf{b}}) u^{2g}=   \frac{1}{(d+1)!} \sum_{\sigma \in S_{d+1}} \prod_{j=1}^{d} \cos\left(u\frac{a_{\sigma(j+1)}(b_{\sigma(1)}+\cdots+b_{\sigma(j)})-b_{\sigma(j+1)}(a_{\sigma(1)}+\cdots+a_{\sigma(j)})}{2}\right)
\end{gather*}
\end{proposition}

In the sequel given a tuple $\textbf{a} = a_1,\cdots,a_n$ define $S_\textbf{a} = -\sum_i a_i$. We write $\overline{\textbf{a}}= (a_1,\cdots,a_n,S_\textbf{a})$. Applying the map $T$ to the equation of Theorem~\ref{buryak21quadratic} and applying Lemma~\ref{lem:varder} yields
\begin{gather}\label{eq : Igd = moyal}
    \sum_{g\geq0} \frac{(-\epsilon^2)^g}{(d+1)!} \sum_{a_1, \ldots, a_{d+1}, b_1, \ldots, b_{d+1}} I_{g,d}(\overline{\textbf{a}},\overline{\textbf{b}})\prod_{j=1}^{d+1} p_{b_j}^{a_j} e^{i(-S_\textbf{a}y-S_\textbf{b}x)}
    =\frac{1}{(d+1)!} T(u \star \cdots \star u).
\end{gather}
\begin{proof}[Proof of Proposition \ref{prop : Igd to cosine}]
    We first make the following claim.
\begin{claim}
 $$T(u \star \cdots \star u) = \sum \prod_{j=1}^{d} \cos\left(u\left(\frac{ a_{j+1}(b_{1}+\cdots+b_{j}) - b_{j+1}(a_{1}+\cdots+a_{j})}{2}\right)\right)\prod_{k=1}^{d+1}p_{b_k}^{a_k}e^{i(-S_{\textbf{a}} y -S_{\textbf{b}} x)} $$ where the first sum on the right hand side is over integers ${a_1,\dots,a_{d+1},b_1,\dots,b_{d+1}}$. The product on the left is $d+1$ times after substituting $u=i\epsilon$.
\end{claim}\label{subclaim5.5}
    We prove Claim \ref{subclaim5.5} by induction on $d$. The base case $d=1$ is done in ~\cite[Proof of Theorem 4.1]{buryak2021quadratic}, after using the dilaton equation to remove the $\psi$ class.
    By definition of the Moyal product there is an expression for $\dagger = T((u \star \dots \star u) \star u)$ as
\begin{equation*}
     \dagger=\sum_{g' \geq 0} \sum_{k_1 + k_2 = 2g'} \frac{(-1)^{k_2}(i\epsilon)^{2g'}}{2^{2g'}k_1 !k_2 !} T(\partial_x^{k_1} \partial_{y}^{k_2}(u \star \dots \star u)) \sum_{a_{d+1},b_{d+1}}(i(b_{d+1}))^{k_2}(i(a_{d+1}))^{k_1}p_{b_{d+1}}^{a_{d+1}}e^{i(a_{d+1}y + b_{d+1}x)}
\end{equation*}
    where we have that the summation range only contributes for even indices because $(u \star \dots \star u) \star u = u \star (u \star \dots \star u)$. The induction step tells us that
    $T(\partial_x^{k_1} \partial_{y}^{k_2}(u \star \dots \star u))$ introduces the factor of $$ (i(\sum_{k=1}^d b_i))^{k_1} \cdot (i(\sum_{k=1}^d a_i))^{k_2} $$ 
which combines with $(-1)^{k_2}(i(b_{d+1}))^{k_2}(i(a_{d+1}))^{k_1}$ to introduce a factor of $$(a_{d+1}(b_1 + \dots + b_d))^{k_1}(-b_{d+1}(a_1 + \dots + a_d))^{k_2}.$$ Now we observe that
\begin{align*}
    &\sum_{g' \geq 0} \sum_{k_1 + k_2 = 2g'} \frac{(i\epsilon)^{2g'}(a_{d+1}(b_1 + \dots + b_d))^{k_1}(-b_{d+1}(a_1 + \dots + a_d))^{k_2}}{2^{2g'}k_1 ! k_2 !} \\
    = &\sum_{g' \geq 0}\frac{\left(i\epsilon(\frac{ a_{d+1}(b_1 + \dots + b_d) - b_{d+1}(a_1 + \dots + a_d)}{2})\right)^{2g'}}{(2g') !} \\
    = &\cos\left(i\epsilon(\frac{a_{d+1}(b_{1}+\cdots+b_{d}) - b_{d+1}(a_{1}+\cdots+a_{d})}{2})\right).
\end{align*}
Note that this is the cosine factor in the statement of the lemma for $j=d$. It follows now that after substituting $u=i\epsilon$ 
$$\dagger = \sum_{a_1,\dots,a_{d+1},b_1,\dots,b_{d+1}} \prod_{j=1}^{d} \cos\left(u\left(\frac{a_{j+1}(b_{1}+\cdots+b_{j}) - b_{j+1}(a_{1}+\cdots+a_{j})}{2}\right)\right)\prod_{k=1}^{d+1}p_{b_k}^{a_k}e^{i(-S_{\textbf{a}} y -S_{\textbf{b}} x)}.  $$
This completes the proof of the claim. Proposition \ref{prop : Igd to cosine} follows from substituting the formula in the claim into equation \eqref{eq : Igd = moyal}, and taking coefficients.
\end{proof}
\section{Completing proof of Theorem \ref{thm:main}}\label{Section : Completing Proof}
We complete the proof of Theorem \ref{thm:main}, by combining Proposition \ref{prop : endgluing} with Lemmas \ref{lem:ThreeValentThroughDR}, \ref{lem:VertexIsDDR} and Proposition \ref{prop : Igd to cosine}. Recall our goal is to prove the following theorem.
\begin{theorem}[Theorem A]
    After the change of variables $q=e^{iu}$ we have the equality
    \begin{equation*}
   \sum\limits_{g \geq 0} {N}_{g,\Delta}^{\textbf{k}} \, u^{2g - 2 + |\Delta^\circ| - \sum_{i} k_i} =\sum\limits_{h \in T_{\Delta,p}^{\textbf{k}}} \prod_{V \in V(\Gamma)} m_V(q)
\end{equation*}
\end{theorem}
\subsection{Generating series as sums over tropical curves}
    The following definitions follow \cite{bousseau2018tropical} and are motivated by Lemma \ref{lemma : almost there}. After fixing an orientation on $\tilde{\Gamma}$ as in Section \ref{def : trivalent contribution}, for a trivalent unpointed vertex $V \in V^{3}(\tilde{\Gamma})$ define $$F^{3\mathsf{u}}_V(u) = \sum_{g \geq 0} N'_{g,V}u^{2g+1} := \sum_{g \geq 0} N_{g,V}w(E_{V}^{\mathsf{in},1}) w(E_{V}^{\mathsf{in},2})u^{2g+1}$$
and for an $m$-valent pointed vertex $V$ with $m\geq3$ define
\begin{equation*}
    F^{\mathsf{mp}}_V(u):= \sum_{g\geq 0} N_{g,V} u^{2g}.
\end{equation*}

\begin{lemma}\label{lemma : almost there}
    $$\sum_{g \geq 0} N_{g,\Delta}^{\textbf{k}}u^{2g-2 + |\Delta^\circ| - \sum_i k_i} = \sum_{h \in T_{\Delta,p}^\textbf{k}} \left(\prod_{V \in V^3(\Gamma)} F^{3\mathsf{u}}_V(u)\right) \left( \prod_{V \in V^{\mathsf{mp}}(\Gamma)} F^{\mathsf{mp}}(u)\right)$$
\end{lemma}
\begin{proof} Observe first that by definition $$\star := \sum_{g \geq 0} N_{g,\Delta}^{\textbf{k}}u^{2g-2 + |\Delta^\circ| - \sum_i k_i} \\
        = \sum_{g \geq 0} \sum_{\tilde{h} \in T^{g,\textbf{k}}_{\Delta,p}} N_{g,\Delta}^{\tilde{h},\textbf{k}}u^{2g-2 + |\Delta^\circ| - \sum_i k_i}$$
By Proposition \ref{prop : endgluing} we know $$N_{g,\Delta}^{\tilde{h},\textbf{k}} = \left(\prod_{E \in E_f(\Gamma)}w(E)\right)\left(\prod_{V \in V^{3}(\tilde{\Gamma})}N_{g_V,V}\right)\left(\prod_{V \in V^{\mathsf{mp}}(\tilde{\Gamma})}N_{g_V,V}\right)$$
whenever $\tilde{h}$ is a parametrised tropical curve with all bivalent vertices having genus 0. Thus we may rewrite
    \begin{align*}
        \star &= \sum_{g \geq 0} \sum\limits_{\substack{\tilde{h} \in T^{g,\textbf{k}}_{\Delta,p} : \\ g_V = 0 \, \forall \, V \in V^{2}(\tilde{\Gamma})}} \left(\prod_{E \in E_{f}(\Gamma)}w(E)\right)\left(\prod_{V \in V^{3}(\tilde{\Gamma})}N_{g_V,V}\right)\left(\prod_{V \in V^{\mathsf{mp}}(\tilde{\Gamma})}N_{g_V,V}\right) u^{2g-2 + |\Delta^\circ| - \sum_i k_i}. 
        \end{align*}
            Since every tropical curve in $T_{\Delta,p}^{g,\textbf{k}}$ has $|\Delta^\circ| - 2 - \sum_i k_i$ unpointed trivalent vertices we learn,
        \begin{align*}
        \star &= \sum_{g \geq 0} \sum\limits_{\substack{\tilde{h} \in T^{g,\textbf{k}}_{\Delta,p} : \\ g_V = 0 \, \forall \, V \in V^{2}(\tilde{\Gamma})}}   \left(\prod_{V \in V^{3}(\tilde{\Gamma})}N'_{g_V,V}u^{2g_V+1}\right)\left(\prod_{V \in V^{\mathsf{mp}}(\tilde{\Gamma})}N_{g_V,V}u^{2g_V}\right). \end{align*}
Now we have,
        \begin{align*}
        \star &= \sum_{h \in T^{\textbf{k}}_{\Delta,p}} \sum_{g \geq 0} \sum\limits_{\substack{g_V : \\ \sum g_V = g}} \left(\prod_{V \in V^{3}(\Gamma)}N'_{g_V,V}u^{2g_V+1}\right)\left(\prod_{V \in V^{\mathsf{mp}}(\Gamma)}N_{g_V,\Delta_V}u^{2g_V}\right) \\
        &= \sum_{h \in T^{\textbf{k}}_{\Delta,p}}    \left( \sum_{g_1 \geq 0}\sum\limits_{\substack{g_1(V) : \\ \sum g_1(V) = g_1}} \prod_{V \in V^{3}(\Gamma)}N'_{g_1(V),V}u^{2g_1(V)+1}\right)\left( \sum_{g_2 \geq 0}\sum\limits_{\substack{g_2(V) : \\ \sum g_2(V) = g_2}}\prod_{V \in V^{\mathsf{mp}}(\Gamma)}N_{g_2(V),V}u^{2g_2(V)}\right) \\
        &= \sum_{h \in T_{\Delta,p}^\textbf{k}} \left(\prod_{V \in V^3(\Gamma)} F^{3\mathsf{u}}_V(u)\right) \left( \prod_{V \in V^{\mathsf{mp}}(\Gamma)} F^{\mathsf{mp}}(u)\right)
    \end{align*} completing the proof.
\end{proof}


\subsection{Formulae for vertex contributions} The right hand side of Lemma \ref{lemma : almost there} resembles the right hand side of Theorem \ref{thm:main}. We have formulae for the summands on the right hand side.

\begin{corollary}\label{Cor:Corollary3u}
Let $V$ be an unpointed trivalent vertex with $\Delta_V^{\circ} = \{v_1,v_2, v_3\}$. After the change of variables $q=e^{iu}$ we have
\begin{equation*}
    F^{3\mathsf{u}}_{V}(u)=(-i)\left(q^{\frac{|v_1 \wedge v_2|}{2}}-q^{-\frac{|v_1 \wedge v_2|}{2}}\right)
\end{equation*}
\end{corollary}
\begin{proof}
By Lemma \ref{lem:ThreeValentThroughDR} we have that 
    \begin{equation*}
        F^{3\mathsf{u}}_V(u) = \sum_{g \geq 0} |v_1 \wedge v_2| \int_{\overline{\mathcal{M}}_{g, 3}} (-1)^{g} \lambda_{g_V} \mathsf{DR}_{g}(\Delta_V^x) \mathsf{DR}_{g}(\Delta_V^y) u^{2g+1}
    \end{equation*}
By combining the dilaton equation with Proposition \ref{prop : Igd to cosine}, or explicitly  \cite[Theorem 2.1]{buryak2021quadratic}, this is equal to 
\begin{align*}
    2\sum_{g \geq 0} \frac{(-1)^{g}|v_1 \wedge v_2|^{2g+1}}{2^{2g+1}(2g+1)!}u^{2g+1} = 2\sin\left(u\left(\frac{|v_1 \wedge v_2|}{2}\right)\right).
\end{align*}
After the substituting $q = e^{iu}$ this gives the result.
\end{proof}

For an $m$-valent pointed vertex with $m \geq 3$ write the multiset 
\begin{equation*}
    \Delta_V^\circ=\{\delta_1^V,\ldots,\delta_m^V\}. 
\end{equation*}
\begin{corollary}\label{Cor:Corollary}
After the change of variables $q=e^{iu}$ we have
\begin{equation*}
    F^{\mathsf{mp}}_V(u)=\frac{1}{2^{m-2}(m-1)!} \sum_{\sigma \in S_{m-1}} \prod_{j=1}^{m-2} \left [\delta_{\sigma(j+1)}^V \wedge \sum_{l=1}^j \delta_{\sigma(l)}^V \right]_+.
\end{equation*}
\end{corollary}

\begin{proof}
    Follows from Proposition \ref{prop : Igd to cosine} by comparing coefficients of powers of $u$, and identifying $d=m-2$.
\end{proof}

Lemma \ref{lem:DiffMult} relates the formulae of Corollary \ref{Cor:Corollary} to the tropical curve multiplicity defined in Section \ref{sec: multiplicities}.
\begin{lemma}\label{lem:DiffMult}
    There is an equality  $$F^{\mathsf{mp}}_V(u)= \frac{1}{(m-1)!}\mu_m(\delta_1^V,\ldots,\delta_m^V)$$.
\end{lemma}
The proof of Lemma \ref{lem:DiffMult} was communicated to us by Thomas Blomme. We record definitions before giving a proof. For $\omega$ in $\Omega_m$ write $\hat{\omega}$ for the unique representative in the group $S_m$ of permutations of $\{1,\ldots,m\}$ which fixes $m$. For $\sigma \in S_{m-1}$ let $\mathsf{A}_\sigma$ be the set of cyclic permutations $\omega\in \Omega_{m}$, such that for any integer $k\leq m-1$, $$ \mathsf{max}\{\hat{\omega}(\sigma(1)),\ldots,\hat{\omega}(\sigma(k))\} - \mathsf{min}\{\hat{\omega}(\sigma(1)),\ldots,\hat{\omega}(\sigma(k))\}=k-1.$$  A cyclic permutation $\omega$ will be recorded by a unique tuple $$(\hat{\omega}(m)=m,\hat{\omega}(m-1), \hat{\omega}(m-2),\ldots,\hat{\omega}(1)).$$ Write $\{-1,1\}^{[j]}$ for the set of functions $\epsilon:\{1,\ldots,j\} \rightarrow \{-1,1\}$. There is a bijective function $$T:\{1,-1\}^{[j]}\rightarrow \mathsf{A}_{\mathsf{id}}.$$ To define $T$ we recursively define $T_1 = (1)$ and then $T_\ell$ is $(\ell,T_{\ell-1})$ if $\epsilon(\ell)=-1$ and $(T_{\ell-1},\ell)$ if $\epsilon(\ell) =1.$ We then set $T(\epsilon)\in {A}_{\mathsf{id}}$ to be $(m,T_{m-2}(\epsilon))$. It is now possible to define $$T_\sigma:\{1,-1\}^{[j]}\rightarrow A_{\mathsf{\sigma}}$$ by composing $T$ with the action of $\sigma$ on $\Omega_n$.
\begin{proof}
    We may rewrite the left hand side as \begin{equation*}
F^{\mathsf{mp}}_V(u)=\frac{1}{2^{m-2}(m-1)!} \sum_{\sigma \in S_{m-1}} \sum_{\epsilon}  q^{\frac{1}{2}\sum_{j=1}^{m-2} \epsilon(j)\left(\delta_{\sigma(j+1)}^V \wedge \sum_{l=1}^j \delta_{\sigma(l)}^V\right)}.
\end{equation*}
Where the second sum is over functions $\epsilon \in \{1,-1\}^{[j]}$.
We now claim that for fixed $\sigma$ $$\sum_{\epsilon}  q^{\frac{1}{2} \sum_{j=1}^{m-2} \epsilon(j)\left(\delta_{\sigma(j+1)}^V \wedge \sum_{l=1}^j \delta_{\sigma(l)}^V\right)} = \sum_{\omega \in A_\sigma}q^{\frac{1}{2}k(\omega)}.$$ 
Indeed, $$\sum_{j=1}^{m-2} \epsilon(j)\left(\delta_{\sigma(j+1)}^V \wedge \sum_{l=1}^j \delta_{\sigma(l)}^V\right) = k(T_{\sigma}(\epsilon)).$$ Thus we may write  \begin{align*}
F^{\mathsf{mp}}_V(u)=&\frac{1}{2^{m-2}(m-1)!} \sum_{\sigma \in S_{m-1}} \sum_{A_\sigma}  q^{\frac{1}{2}k(\omega)}\\  =& \frac{1}{2^{m-2}(m-1)!} \sum_{\omega \in \Omega} q^{\frac{1}{2}k(\omega)} \mathsf{card}\{\sigma| \omega \in A_\sigma\}.
\end{align*} The number $\mathsf{card}\{\sigma| \omega \in A_\sigma\}$ is independent of $\omega$ and so we assume $\omega = \mathsf{id}$. Then note $\mathsf{card}\{\sigma| \omega \in A_\sigma\} = 2^{m-2}$. We conclude  \begin{equation*}
F^{\mathsf{mp}}_V(u)= \frac{1}{(m-1)!} \sum_{\omega \in \Omega} q^{\frac{1}{2}k(\omega)}
\end{equation*} and the result is proved.
\end{proof}

\subsection{Finishing the proof} In this section we will write $\mu_m(\delta_1^V,\ldots,\delta_m^V) = \mu_V(q)$.
\begin{proof}[Proof of Theorem \ref{thm:main}]
    Substituting Corollary \ref{Cor:Corollary} and Lemma \ref{lem:DiffMult} into Lemma \ref{lemma : almost there} then rearranging we learn, 
\begin{align*}
     \sum_{g \geq 0} N_{g,\Delta}^{\textbf{k}}u^{2g-2 + |\Delta^\circ| - \sum_i k_i} &= \sum_{h \in T_{\Delta,p}^\textbf{k}}\left(\prod_{V \in V^3(\Gamma)} (-i)(q^{\frac{m(V)}{2}} - q^{-\frac{m(V)}{2}})\right) \cdot \left( \prod_{V \in V^{\mathsf{mp}}(\Gamma)} \frac{1}{(\mathsf{val}_V-1)!}\mu_V(q)\right) \\ 
     &=  \sum_{h \in T_{\Delta,p}^\textbf{k}}  \prod_{V \in V(\Gamma)} m_V(q) \\
     &=  \Ntrop(q).
\end{align*}
This completes our proof.\end{proof}

\bibliography{bibliography}
\bibliographystyle{siam}

\footnotesize

\end{document}